\documentclass{amsart} \usepackage{amssymb,amsfonts,amscd}
\newtheorem{theorem}{Theorem}[section]
\newtheorem{lemma}[theorem]{Lemma}
\newtheorem{corollary}[theorem]{Corollary}
\newtheorem{proposition}[theorem]{Proposition}
\theoremstyle{remark}

\theoremstyle{definition}

\numberwithin{equation}{section} \makeatother

\newcommand{\norm}[1]{\left\|#1\right\|}

\newcommand{\Ker}[1]{{\rm Ker}(#1)}

\DeclareMathOperator{\spn}{Span} 
 
\DeclareMathOperator{\Kdb}{{\mathbb K}}

\DeclareMathOperator{\Cdb}{{\mathbb C}}
\DeclareMathOperator{\Rdb}{{\mathbb R}}
\DeclareMathOperator{\Ddb}{{\mathbb D}}

\DeclareMathOperator{\Ndb}{{\mathbb N}}

\begin{document}

\title[Ideals and HSA's in operator algebras]{Ideals and hereditary subalgebras  in operator algebras}
\author{Melahat Almus}
\author{David P. Blecher}
\author{Charles John Read}

\address{Department of Mathematics, University of Houston, Houston, TX
77204-3008}
 \email[Melahat Almus]{almus@math.uh.edu}
\address{Department of Mathematics, University of Houston, Houston, TX
77204-3008}
 \email[David P.
Blecher]{dblecher@math.uh.edu}
\address{Department of Pure Mathematics,
University of Leeds,
Leeds LS2 9JT,
England}
 \email[Charles John Read]{read@maths.leeds.ac.uk}
\thanks{*Blecher was partially supported by a grant  from
the National Science Foundation.}
\date{Revision of June 13, 2012}
\begin{abstract}
This paper may be viewed as having two aims.
First, we continue our study of algebras of operators on a Hilbert space
which have a contractive approximate identity, this
time from a more Banach algebraic point of view.  Namely, we
mainly  investigate topics concerned with the ideal structure, and hereditary subalgebras (HSA's), which are
in some sense  
generalization of ideals.   Second, we study properties of operator algebras which are hereditary subalgebras in their bidual, or equivalently which are `weakly compact'.  We also give several examples answering natural questions that 
arise in such an investigation.
\end{abstract}

\maketitle

\section{Introduction}  For us, an operator algebra is a norm closed algebra of operators on a Hilbert space.
This paper may be viewed as having two aims.
First, we continue our study of the structure of operator  algebras 
which have a contractive approximate identity (cai), this time from a slightly more  Banach algebraic point of view.  Indeed 
this paper may be seen as a collection of general results growing out of topics raised in \cite{ABS}
 concerning ideals  and  {\em hereditary subalgebras} (HSA's) of operator algebras.
We recall that HSA's are in some sense  
a generalization of ideals (a HSA $D$ in $A$ must satisfy $DAD \subset D$) .   Some of these general results are
of a technical nature, and so this paper should serve in part as a repository that will be  useful in later development of  the
themes of interest here.  
 Second, we study properties of operator algebras which are hereditary subalgebras in their bidual (this is equivalent to $A$ being `weakly compact',
see e.g.\  Lemma \ref{cohwc} below).  We give several examples answering natural questions that 
arise in such an investigation.  For example, questions involving semisimplicity or semiprimeness (we recall
that an algebra $A$  is semisimple if its Jacobson radical $J(A) = (0)$, and is semiprime if $(0)$ is the only 
(closed) ideal with square $(0)$).

In Section 2 (resp.\  Section 4) of the paper we present some general results about ideals (resp.\ HSA's)
in 
operator algebras.   In Section 3 we discuss adjoining a square root to an operator algebra, and use this to 
answer some natural questions.
We also  discuss in Section 3 whether for an approximately unital operator algebra being   semisimple (resp.\ semiprime, radical)
implies or is implied by $A^{**}$  being   semisimple (or semiprime, radical).  (We do 
not think we know yet  if $A^{**}$ semisimple implies $A$ semisimple  if $A$ is noncommutative.)  
In Section 5 we study operator algebras $A$ which are hereditary subalgebras of their bidual, which
as we said above is equivalent to the multiplication $x \mapsto axb$ being weakly compact on $A$
for all $a, b \in A$.  Some of the properties of such algebras are similar to operator algebras which are one-sided ideals in their bidual, which were studied in \cite{Sharma,ABS}.  We also study
the more 
general class of algebras that we call {\em nc-discrete}, which means that all the open projections are 
also closed (or equivalently lie in the multiplier algebra $M(A)$).    Any compact operator algebra is
a HSA in its bidual; and we show that any operator algebra which  is
a HSA in its bidual is nc-discrete. Neither of these two implications are reversible though, as
may be seen in  Examples \ref{ell2}, 
\ref{ell1}, and Theorem \ref{wcnc}.   Indeed in  Section 6 we 
present examples of operator algebras exhibiting  various properties illustrating  the topics
 of interest in this paper.   In particular we give what is as far as we can see  the 
first explicit example in the literature of an interesting (i.e.\  not reflexive in the Banach space sense) 
commutative algebra 
whose multiplication is weakly compact but not compact.   In Section 7 we discuss the diagonal of a quotient
algebra.

Some of the topics in the earlier parts of this paper and in Section 7 were investigated in the first author's Ph.D.\ thesis \cite{AlTh}.
Some related results and topics may be found there too.   Others of our results related to HSA's have been presented at various venues in 2010.

We now turn to notation and more precise definitions.
The reader is referred for example to \cite{BLM,BHN,BRead} for more details on
some of the topics below if needed.   By an {\em ideal} of an operator algebra $A$ we shall
always mean a closed two-sided ideal in $A$.   For us a {\em projection}
is always an orthogonal projection, and an {\em idempotent} merely
satisfies $x^2 = x$. If $X, Y$ are sets, then $XY$ denotes the
closure of the span of products of the form $xy$ for $x \in X, y \in
Y$.     We recall that by a theorem due to  Ralf Meyer,
every operator algebra $A$ has a unique unitization $A^1$ (see
 \cite[Section 2.1]{BLM}). Below $1$ always refers to
the identity of $A^1$ if $A$ has no identity.  If $A$ is a nonunital
operator algebra represented (completely) isometrically on a Hilbert
space $H$ then one may identify $A^1$ with $A + \Cdb I_H$.   The
second dual $A^{**}$ is also an operator algebra with its (unique)
Arens product, this is also the product inherited from the von Neumann
algebra $B^{**}$ if
$A$ is a subalgebra of a $C^*$-algebra $B$.  Meets and joins in
$B^{**}$ of projections in $A^{**}$ remain in $A^{**}$, as can be
readily seen for example by inspecting some of the classical
formulae for meets and joins of Hilbert space projections,
or by noting that these meets and joins may be computed in the biggest
von Neumann algebra contained inside $A^{**}$. Note that
$A$ has a cai iff $A^{**}$ has an identity $1_{A^{**}}$ of norm $1$,
and then $A^1$ is sometimes identified with $A + \Cdb 1_{A^{**}}$.
In this case the multiplier algebra $M(A)$ is identified with the
idealizer of $A$ in $A^{**}$ (that is, the set of elements
$\alpha\in A^{**}$ such that $\alpha A\subset A$ and $A
\alpha\subset A$).

The set
 of compact operators on a Hilbert space is often called an
{\em elementary $C^*$-algebra}.  We call a $c_0$-direct sum of elementary $C^*$-algebras
 an {\em annihilator $C^*$-algebra.}

 The {\em diagonal} $\Delta(A)$ is defined to be $A \cap A^*$,
it is a $C^*$-algebra which is well defined independently of  the particular (completely isometric) representation of $A$. 
Most of our algebras 
and ideals  are {\em approximately unital}, i.e.\ have a contractive approximate identity (cai), although
for some results this is probably not necessary.
We recall that an {\em r-ideal} is a right ideal with a left cai, and an {\em $\ell$-ideal} is a left ideal with a right cai.
We say that an operator algebra $D$ with cai, which is a subalgebra of
another operator algebra $A$, is a HSA (hereditary subalgebra)
in $A$, if $DAD \subset D$.    See
\cite{BHN} for the theory of HSA's (a few more results may be found in \cite{ABS,BRead}).
HSA's in $A$ are in an order preserving,
bijective correspondence with the r-ideals in $A$, and also with the
{\em open projections} $p \in A^{**}$, by which we mean that there
is a net $x_t \in A$ with $x_t = p x_t p \to p$ weak*.  These are
also the open projections $p$ in the sense of Akemann \cite{Ake2} in $B^{**}$, where $B$ is a $C^*$-algebra containing $A$, such that
$p \in A^{\perp \perp}$.   The complement (`perp') of an open projection is called a {\em closed projection}.  We spell out some of the correspondences above:
if $D$ is a HSA in $A$, then $DA$ (resp.\ $AD$) is the matching  r-ideal (resp.\ $\ell$-ideal),
and $D = (DA)(AD) = (DA) \cap (AD)$.
The weak* limit of a cai for $D$, or of a left cai for an r-ideal, is
an open projection, and is called the {\em support projection}.  
Conversely, if $p$ is an open projection in $A^{**}$, then
$pA^{**} \cap A$ and $pA^{**}p \cap A$ is the matching r-ideal and HSA pair
in $A$.  

  It is a well-known fact that if $J$ is an ideal of an operator algebra $A$, then the quotient algebra $A/J$ is isometrically isomorphic to an operator algebra \cite[Proposition 2.3.4]{BLM}. 
  Of course there is a `factor theorem':
 if $u: A \to B$ is a completely bounded homomorphism between 
operator algebras, and if $J$ is an ideal in $A$ contained in Ker$(u)$, then the 
canonical map $\tilde{u}: A/J \to B$ is also completely bounded with completely bounded norm 
$\norm{\tilde{u}}_{cb}=\norm{u}_{cb}$. If $J = \Ker{u}$, then $u$ is a complete quotient map if and only if $\tilde{u}$ is a completely isometric isomorphism.

Let $A$ be an operator algebra.  The set
${\mathfrak F}_A = \{ x \in A : \Vert 1 - x \Vert \leq 1 \}$ equals
$\{ x \in A : \Vert 1 - x \Vert = 1 \}$ if $A$ is nonunital, whereas
if $A$ is unital then ${\mathfrak F}_A = 1 + {\rm Ball}(A)$.  Many properties of ${\mathfrak F}_A$ are 
developed in \cite{BRead,BReadII}.  If $A$ is a closed subalgebra of an operator algebra $B$
then it is easy to see, using the
uniqueness of the unitization, that ${\mathfrak F}_A = A \cap {\mathfrak F}_B$.

We write $J(A)$ for the Jacobson radical (see e.g.\ \cite{Pal}).  It is a fact in 
pure algebra that an algebra is semiprime (resp.\ semisimple) iff
its unitization is semiprime (resp.\ semisimple).  Indeed
$J(A) = J(A^1)$ (see \cite[4.3.3]{Pal}).    
  One trap to beware of is that
 the $C^*$-algebra generated by a HSA $D$ in an operator algebra
$A$ need not be a
HSA in a $C^*$-algebra generated by $A$.  In particular $C^*(D)$ need not be an
HSA in $C^*_{\rm max}(A)$ or in $C^*_{\rm e}(A)$.
An example is ${\mathcal U}(M_2)$, the subalgebra of $M_2(A)$ with $0$ in the 
$2$-$1$-entry, and scalar multiples of the identity on the main diagonal, 
in the case $A = M_2$.

\section{General results on ideals in operator algebras}

The first two results that follow are obvious, and follow from the analogous 
results for general operator spaces.

\begin{theorem}[First Isomorphism Theorem]\label{FIT}
Let $u : A \to B$ be a complete quotient map which is a homomorphism between operator algebras. Then, $\Ker{u}$ is an ideal in $A$ and $A / \Ker{u} \cong 
B$ completely isometrically isomorphically. Conversely, every ideal of $A$ is of the form $\Ker{u}$ for a complete quotient map $u:A \to B$, where $A$ and $B$ are operator algebras.
\end{theorem}

\begin{theorem} [Second Isomorphism Theorem] \label{SIT}
Let $A$ be an approximately unital operator algebra, let
$J$ be an ideal in $A$, and suppose that $I$ is an ideal in $J$.
  Then, $(A/I) / (J/I) \cong  A/J$ completely isometrically isomorphically
(as operator algebras).
\end{theorem}

\begin{theorem} [Third Isomorphism Theorem] \label{TIT}
Let $A$ be an approximately unital operator algebra,
 and suppose that $J$ and $K$ are ideals in $A$, where $J$ has a cai. Then, $J / (J \cap K) \cong (J+K) / K$ completely isometrically isomorphically. In particular, $(J+K)/K$ is closed.
\end{theorem}
\begin{proof} Note that by \cite[Proposition 2.4]{Dixon}, $J+K$ is closed.
Define a map $u : J/(J \cap K) \rightarrow (J+K)/K$ by $u (j+ J \cap K) = j + K$. This is a well-defined map and $u$ is one-to-one since $\Ker{u} = (0_{J/(J \cap K)})$. Moreover, $u$ is onto since $x +K \in (J+K)/K$ implies that $x=j+k$, where $j \in J, k \in K$ and $x+ K= j+K = u(j+ J \cap K)$.

Since $\inf{ \{ \norm{j+k}: k \in K \}} \leq \inf{\{ \norm{j + k}: k \in J \cap K\} }$, $u$ is a contraction. Let $(e_t)$ be the cai for $J$ and let $k \in K$. Then,
\[\norm{j+k} \geq \norm{e_t j + e_t k} \geq \norm{e_t j + J \cap K} . \] 
After taking the limit, we get $\norm{j+ J \cap K} \leq \norm{j+k}$, and 
so $\norm{j + J \cap K} \leq \norm{j+K}$. Hence, $u$ is an isometry. Similarly, $u$ is a complete isometry.
\end{proof}

For Banach algebras the `Correspondence Theorem' states that for
a Banach algebra $A$ and a closed ideal $J$ in $A$, every closed subalgebra 
$K$ of $A/J$ is of the form $I/J$, where $I$ is a closed subalgebra of $A$ with $J \subset I \subset A$.
Also, every ideal $K$ of $A/J$ is of the form $I/J$, where $I$ is an ideal of $A$ with $J \subset I \subset A$.  Indeed $I = q^{-1}(K)$ where $q : A \to  A/J$ is
the canonical map.

\begin{theorem} \label{COTau}
Let $A$ be an approximately unital operator algebra, let $I$ be an approximately 
unital ideal in $A$ and let $J$ be an approximately unital ideal in $I$. Then, 
$I/J$ is an approximately unital ideal in $A/J$.
Conversely, every approximately unital ideal of $A/J$ is of the form $I/J$, where $I$ is an approximately unital ideal in $A$ with $J \subset I \subset A$.
\end{theorem}
\begin{proof}
The first assertion is easy.
The second assertion follows from \cite[Proposition 3.1]{BR} (as in e.g.\  
\cite[Section 6]{BRead}, where the analogue of the above 
result is proved for HSA's and certain one-sided ideals).   \end{proof}

{\bf Remark.} Note that since \cite[Proposition 3.1]{BR} is also valid for Arens regular Banach algebras, the previous result can be stated for such Banach algebras.
Similarly for Corollary \ref{cor-sp} below.

\begin{lemma}  \label{apasp}  Suppose that $A$ is an operator algebra such that $A^{**}$ is semiprime,
 and that $J$ is a closed ideal in $A$ such that $J^2$ has a cai.  Then $J = J^2$.
\end{lemma}  

\begin{proof}   Let $p$ be the support projection
of $J^2$ in $A^{**}$.  We have $J^2 (1-p) = 0$.   If $\zeta, \eta \in J^{\perp \perp}$
and if $a_t \to \eta$ weak* and $b_s \to \zeta$, then since $a_t b_s (1-p) = 0$
we have $a_t \zeta (1-p)  = 0$ and $0 = \eta \zeta (1-p) = \eta (1-p) \zeta (1-p)$.
So $(J^{\perp \perp} (1-p))^2 = (0)$.  Since $A^{**}$ is semiprime we have $J^{\perp \perp} (1-p)
= (0)$, so that $J^{\perp \perp}$ is unital, so that $J$ has a cai.   \end{proof}

\begin{corollary} \label{cor-sp} Suppose that
 $A$ is an approximately unital operator algebra, 
 and that $J$ is an approximately unital ideal in $A$.
Suppose that  $J$ has  the property that if $I$ is an ideal with square $J$ then $I = J$.  Then 
 $A/J$ is semiprime.  In particular, if $A^{**}$ is semiprime then 
$A/J$ is semiprime,  for every approximately unital ideal $J$ in $A$.
\end{corollary}
\begin{proof} Let $K$ be an ideal in $A/J$ such that $K^2 =(0)_{A/J}$.  There exists an 
ideal $I$ in $A$ such that $J \subset I \subset A$ and $K = I/J$ (namely, 
the inverse image of $K$ in $A$, see Theorem \ref{COTau}). Since $K^2 =(I/J)(I/J)= I^2/J = (0)_{A/J}$, we conclude that $I^2 = J$.      Under our hypotheses this forces $I = J$ (the `In particular'
assertion uses Lemma \ref{apasp} here).  
That is, $K = I/J = (0)_{A/J}$, and so $A/J$ is semiprime.
\end{proof}

{\bf Remarks.}  1) \ In view of the last results, and independently, it is of interest to know whether 
every approximately unital ideal $J$ in  a semisimple or semiprime algebra $A$ has the property 
that  if $I$ is an ideal with square $J$ then $I = J$.  We will see in the next section
after Lemma \ref{ssq}  that this is false.  

2) \ Algebras whose square (or $n$th power) is approximately unital are discussed  in \cite[Section 3]{BReadII}.

\begin{proposition} \label{ab}
Let $A$ be a Banach algebra with no nonzero left
annihilators, and let $\{I_\alpha\}$ be an increasing family of
ideals in $A$ such that $A = \overline{\cup{I_\alpha}}$. If each $I_\alpha$ is
semiprime (resp. semisimple), then $A$ is semiprime (resp. semisimple).
\end{proposition}

\begin{proof}   For each $\alpha$, assume that $I_\alpha$ is semiprime.
Let $J$ be an ideal in $A$  with $J^2 = (0)$. Then $(J \cap I_\alpha)^2 = (0)$
for each $\alpha$.
 Since $I_\alpha$ is semiprime, $J \cap I_\alpha =(0)$.
Hence $J I_\alpha \subset  J \cap I_\alpha =(0)$ for each $\alpha$, so that $J A = (0)$ and
$J= (0)$. Hence $A$ is semiprime.

Now suppose that each $I_\alpha$ is semisimple. Notice that $J=\text{Rad}(A)$
is an ideal in $A$ and $J \cap I_\alpha$ is an ideal in
$I_\alpha$ for each $\alpha$. Then $J \cap I_\alpha = \text{Rad}(I_\alpha) =(0)$
by  \cite[Theorem 4.3.2]{Pal}. Hence as in the
first paragraph, $J= (0)$ and $A$ is semisimple.
    \end{proof}

\begin{proposition} Suppose that $A$ is an operator algebra with a bai, and that
$I$ and $J$ are  ideals in $A$. If $A/I \cong  A/J$ isomorphically as $A$-bimodules, then $I = J$.
\end{proposition}
\begin{proof} If  $A$ is unital then this
follows from the analogous result in pure algebra.
If $A$ contains a bai, then its bidual $A^{**}$ is unital. If $\pi: A/I \to A/J$ is an $A$-bimodule isomorphism,  then $\pi^{**}$ maps $(A/I)^{**} \cong A^{**}/I^{\perp \perp}$ into $(A/J)^{**} \cong A^{**}/J^{\perp \perp}$. Moreover, $\pi^{**}$ is an $A^{**}$-bimodule isomorphism by the separate weak$^*$-continuity of the Arens product.
 Since
$A^{**}$ is unital, by the unital case we have $I^{\perp \perp} = J^{\perp \perp}$. Thus $I = A \cap I^{\perp \perp} = A \cap J^{\perp \perp} = J$.
\end{proof}

{\bf Remark.} Note that the previous proposition is not true for general operator algebras; the existence of a bai is needed. For example, let $A=\spn(x,y)$ where $xy= x^2 = y^2 =0$. If $I=\spn(x)$ and $J=\spn(y)$, then $A/I \cong A/J$ as $A$-bimodules, but $I \neq J$. 

\section{Example: adjoining a root to an algebra}

In this section we show how to create examples of operator algebras by adjoining a root.  We then use
this to answer several basic questions regarding operator algebras with cai.

If $A$ is an algebra, and $S$ is in the center of $A$, we define an algebra $$A_S = \left\{ \left[
\begin{array}{ccl}  x & y \\ Sy & x \end{array} \right] : x \in A, y \in A^1 \right\} \subset M_2(A^1) .$$ 
We identify $A$ with the main diagonal of this algebra, and we set $T$ to be the matrix with rows
$0, 1$ and $S, 0$.  Then any element of $A_S$ may be written as $x + y T$ for $x \in A, y \in A^1$.
In this notation, $T^2 = S$, and so now $S$ has a square root even if it did not have one before.
A good example to bear in mind is the case that $A$ is the approximately unital  ideal in  the
disk algebra $A(\Ddb)$ of functions vanishing at $1$, and $S = z(1-z) \in A$, which has no root in $A$.  

It is obvious that if $A$ is an operator algebra
then so is $A_S$, and  if $A$ is commutative  then so is $A_S$.  If $A$ has a cai but no identity then
$A_S$ has no cai, but $A^2_S = \{ x + y T : x, y \in A \}$ does have a cai.

We say that an element $a$ in an algebra $A$ has no rational square root if there exists no 
$b, c \in A$, with $a c^2 = b^2 \neq 0$.  

\begin{lemma} \label{ssq}  Suppose that $A$ is a commutative semisimple algebra,
 and $S \in A, S \neq 0$.  If $A$ is an integral domain then $A_S$ is semiprime.
On the other hand, if $A$ is semisimple and $S$ has no rational square root, and is not a divisor of zero,
then  $A_S$ is semisimple.  
\end{lemma}  \begin{proof}  If $(x + yT)^2 = x^2 + y^2 S + 2xy T = 0$, and $A$ is an integral domain, then 
$x = 0$ or $y = 0$.  Since $x^2 + y^2 S = 0$ we have $x = y = 0$.   

 In the  semisimple case, suppose that all characters of $A_S$ vanish at $x + yT
\in A_S$.   If $\chi$ is a character of $A^1$, define $\chi'(a + bT) = \chi(a) + \alpha \chi(b)$ 
for $a \in A, b \in A^1$, where $\alpha$ is a square root of $\chi(S)$.   This defines a character on 
$A_S$, and so    we have $\chi(x) + \alpha \chi(y) = 0$.  Thus
$x^2 - S y^2$ is in the kernel of every   character of $A^1$, so that $S y^2 = x^2$ .   Thus
$x^2 = 0$, and since $A$ is semiprime we have $x = y = 0$.  
\end{proof} 

In our disk algebra example mentioned above, 
the element $z(1-z) \in A$ has no rational square root.  To see this note that of course $1-z$ does,
and so we are asking if $z g^2 = h^2$ is possible with $g, h \in A(\Ddb)$.  By Riemann's theorem
in basic complex analysis
this equation implies that $h/g$ has an analytic continuation $k$ to $\Ddb$ such that $k(z)^2 = z$ on 
$\Ddb$, which is well known to be impossible ($2 k(z) k'(z) = 1$ so $k'$ is unbounded at $0$).
We deduce from Lemma \ref{ssq} that $A_S$ is semisimple if $S = z (1-z)$.  Here 
$A_S$ is a semisimple commutative operator algebra
with no cai, but $A_S^2$ has a cai.  This solves the  question posed in  the Remark after 
Corollary \ref{cor-sp}  in the negative.  

In \cite{ReadDaws} it is shown that  semisimple $B(\ell^p)$ fails to have a semisimple
second dual if $p \neq 2$.  This can also happen for operator algebras: 

\begin{proposition} \label{isnotsp}   Let $A$ be an operator algebra. 
\begin{itemize}  \item [(1)]  If  $A^{**}$ is semiprime (resp.\  radical, semisimple and commutative) then $A$ is semiprime (resp.\  radical,
semisimple).
 \item [(2)] If $A$ is semiprime (resp.\ approximately unital and  radical, unital and semisimple)
then  $A^{**}$  need not be semiprime (resp.\ radical, semiprime and hence not semisimple). 
\end{itemize}   \end{proposition}

\begin{proof}   (1) \   If $A$ is commutative and $A^{**}$ is semisimple, then $A$ is semisimple by e.g.\
\cite[Proposition 2.6.25 (iv)]{Dal}.
If $A^{**}$ is semiprime and if $J^2 = (0)$ in $A$
then $(J^{\perp \perp})^2 = (0)$ in $A^{**}$, so that 
$(0) = J^{\perp \perp} = J$.   So $A$  is semiprime.

It follows from \cite[Proposition 2.6.25 (iii)]{Dal} 
that if $A^{**}$ is radical then $A$ is radical.  

(2) \ If $A$ is radical then $A^{**}$ is  not radical (indeed $A^{**}$ is unital). 

Suppose that the second dual of every  
unital semiprime operator algebra $A$ was semiprime.  Then by Lemma  \ref{apasp}, if 
 $J$ is an ideal in 
a unital semisimple operator algebra such that $J^2$ has a cai, then $J$ has a cai.
However this is not true, as may be seen from the disk algebra example two paragraphs above (one may take $J = A_S, A = J^1$ here).  Hence 
the second dual of a commutative unital  semisimple
 operator algebra need not be semiprime.    \end{proof}  

 We shall see
in Corollary \ref{sssp} that the situation in the last result improves if $A$ is a HSA in its bidual.  

\section{General facts about HSA's}

\begin{proposition} \label{qi} If $D$ is a HSA in an operator algebra $A$,
then $x \in D$ is quasi-invertible in $D$ iff $x$ is quasi-invertible in $A$.
Thus $\sigma_D(x) \setminus \{ 0 \} = \sigma_A(x)  \setminus \{ 0 \}$.
In particular, $D$ is a spectral subalgebra of $A$ in the sense of e.g.\ 
{\rm \cite[p.\ 245]{Pal}}.
\end{proposition}

\begin{proof}
By \cite[Proposition 2.6.25]{Dal}, $x$ is quasi-invertible in $D$
(resp.\ $A$) iff $x$ is quasi-invertible in $D^{**}$ (resp.\ $A^{**}$).
Now $D^{**} = pA^{**} p$, where $p$ is the support projection of $D$.
It is a simple algebraic exercise that in an algebra $A$ with an idempotent
$e$, an element of $eAe$ is quasi-invertible 
 in $eAe$ iff it is quasi-invertible in $A$.  So
$x$ is quasi-invertible in $D$
iff $x$ is quasi-invertible in $D^{**} = pA^{**} p$,
iff $x$ is quasi-invertible in $A^{**}$, and hence iff
$x$ is quasi-invertible in $A$.

The last assertion follows from the usual formulation
of the spectrum in terms of quasi-invertible elements.
\end{proof}

The following answers a question posed in \cite{ABS}.
The first assertion is well known with HSA's replaced by ideals 
\cite{Pal}.   
 
\begin{theorem}  \label{sspas}   If $D$ is a HSA in an operator algebra $A$,
then $J(D) = D \cap J(A)$.
In particular, semisimplicity passes to HSA's.
\end{theorem}

\begin{proof}
Suppose that $A$ is an operator algebra and 
$D$ is a HSA in $A$.  We recall that $J(A)$ may be characterized
 (see e.g.\ 
\cite{Pal}) as the set of $a \in A$ with $r(ab) = 0$ for all $b \in A^1$. 
Here $r(\cdot)$ denotes the spectral radius.
  Let $x \in J(D)$ then
since $J(D)$ is a nondegenerate $D$-module, by Cohen's factorization there
exists $d \in D, y \in J(D)$ with $x = d y$.  Now $y f_t \, a d \in J(D)$
for all $a \in A^1$,
where $(f_t)$ is a cai for $D$ (since $D$ is a HSA in $A^1$).  Since
$J(D)$ is closed we have $y a d \in J(D)$.  Thus $0 = r(y a d) =
r(d y a) = r(xa)$ for all $a \in A^1$.  Hence $x
\in J(A)$.  So $J(D) \subset D \cap J(A)$.
The converse follows from \cite[Theorem 4.3.6 (c),(e)]{Pal}:   
 if $x \in D \cap J(A)$, then $A^1 x$ consists of 
quasi-invertibles in $A$.  Hence $D^1 x$ consists of
quasi-invertibles in $A$, hence of
quasi-invertibles in $D$ by Proposition \ref{qi}.    So $x \in J(D)$.  
\end{proof}

We have a generalization of the last result:

\begin{corollary} \label{innp}  Suppose that $D$ is a HSA in an operator algebra $A$,
and that $I$ is an approximately unital ideal in $A$.
Then \begin{itemize} \item [(1)]
$D \cap I = DID$ is a HSA in $A$, and $J(D \cap I) = J(D) \cap J(I)$.
\item [(2)] $(D \cap I)^{\perp \perp} = D^{\perp \perp} \cap I^{\perp \perp}$. 
\item [(3)] $D + I$ is closed, and is a HSA in $A$.
  \end{itemize} \end{corollary}  

\begin{proof}  (1) \ 
We have that $(D \cap I) A (D \cap I) \subset (DAD) \cap (IAI) \subset D \cap I$.  
Note that $D I D \subset I \cap D$.  Conversely,
since $D$ has a cai we have $I \cap D \subset D I D$.
So $D I D = I \cap D$.    If $(f_s)$ is a cai for $I$ and $(e_\lambda)$ is a cai for $D$, 
then
$(e_\lambda f_s e_\lambda)$ is easily seen to yield a cai for $D I D$, 
by routine techniques.  so $I \cap D = D I D$  is a HSA in $A$.
By Theorem \ref{sspas}  we have $$J(D \cap I) = D \cap I \cap J(A) 
= D \cap  J(A) \cap I \cap J(A) = J(D) \cap J(A) $$   as
desired.     

(3) \ Write $(e_\lambda)$ for the cai of $D$.
If $r \in I$ then $e_\lambda r e_\lambda \in D \cap I$.
Moreover if $a \in D$ then
$$\Vert a- r \Vert \geq \Vert e_\lambda a e_\lambda - e_\lambda r e_\lambda \Vert
\geq  \Vert a  - e_\lambda r e_\lambda \Vert - \Vert a  - e_\lambda  a e_\lambda  \Vert.$$
Also, $\Vert a  - e_\lambda r e_\lambda \Vert \geq \Vert a + (I \cap D) \Vert$.
The above constitute the modifications of the proof of
 \cite[Proposition 2.4]{Dixon} that need to be made so that as in that proof we 
may deduce that $D + I$ is closed.
Clearly $(D + I) A (D + I) \subset
D A D + I = D + I$.  

(2) \ This follows from (3) and the fact 
from e.g.\ \cite[Appendix A.3, A.5]{BZ} that for 
closed subspaces $E,F$ of any Banach space $X$, 
$$(E \cap F)^{\perp \perp} = (E^{\perp} + F^{\perp})^{\perp}
= E^{\perp \perp} \cap F^{\perp \perp},$$ if $E + F$ is closed, or equivalently,
 if $E^{\perp \perp} + F^{\perp \perp}$ is closed. 
   \end{proof}

{\bf Remark.}   If $I$ is any ideal in a HSA $D$ of
an operator algebra $A$, and if $I \subset J(A)$, then $J(D/I) = J(D)/I$.  
This follows from Theorem \ref{sspas} and  
 \cite[Theorem 4.3.2 (b)]{Pal}.

 \section{Algebras that are HSA's in their bidual}

That is, $A$ has a cai, and $A A^{**} A \subset A$.  
We write $M_{a,b} : A \to A : x \mapsto axb$,
where $a, b \in A$.   Recall that a Banach algebra
is {\em compact} if the map $M_{a,a}$ is compact
for all $a \in A$.  We say that $A$ is {\em weakly compact} if 
$M_{a,a}$ is weakly compact for all $a \in A$. 
We are concerned here mostly with operator algebras $A$
that are HSA's in their bidual.
That is, $A$ has a cai, and $A A^{**} A \subset A$.  For algebras $A$ that do not 
have a cai, one could pass to the algebra $A_H$ described in \cite{BReadII}, 
the biggest subalgebra with a cai.   

\begin{lemma} \label{cohwc}  An operator algebra $A$ with cai is a HSA in its bidual
iff the map $M_{a,b} : A \to A : x \mapsto axb$ is weakly compact
for all $a, b \in A$, and iff
$A$ is weakly compact in the sense just defined.

Similarly, $M_{a,b}$ is compact on $A$ 
for all $a, b \in A$ iff $A$ is compact.

If in addition $A$ is commutative, then $A$ is compact (resp.\ weakly compact) iff 
multiplication by $a$ is   compact (resp.\ weakly compact) on $A$ for all $a \in A$.
\end{lemma}

\begin{proof}  The first `iff' follows by  basic functional analysis (namely, the
well known fact that
an operator
$T : X \to Y$ is  weakly compact iff $T^{**}(X^{**}) \subset Y$).
To see the second and third `iff' we use the 
fact that the compact (resp.\ weakly compact) operators constitute a norm closed ideal.
From this, first, if $(e_t)$ is a cai for $A$
and $M_{e_t,e_t}$ is  compact (resp.\ weakly compact), then so is
$M_{ae_t,e_tb}$  for all
$a, b \in A$.
 Second, $M_{a,b}$ is  compact (resp.\ weakly compact)
since  $M_{ae_t,e_tb} \to M_{a,b}$ in norm.  \end{proof}

Clearly then compact operator algebras are HSA's in their bidual.   It is easy to find Banach space reflexive 
examples showing that the converse is not true (see Example \ref{ell2}).
Note that the class of unital operator algebras
which are HSA's in their bidual, is the same as
the class of
 unital operator algebras which are Banach space reflexive. 
 It is of interest to find nonreflexive weakly compact algebras which are not compact,
and we shall do this  later in 
 Subsection \ref{weight}.    
 In this connection we remark that semisimple annihilator Banach algebras in the sense of
\cite[Chapter 8]{Pal} are compact, and are {\em ideals} in their bidual
\cite[Corollary 8.7.14]{Pal}.

\medskip

{\bf Remark.}  In any commutative operator algebra $A$, two  natural   ideals to consider 
are those constituting the elements $a \in A$ with multiplication by $a$ being compact or weakly compact on $A$.

\medskip

The property of being a HSA in the bidual passes to subalgebras and quotients:

\begin{lemma} \label{hsher}  Let  $A$ be an operator algebra
which is weakly compact.  If $B$ is a closed subalgebra
of $A$, then $B$ is  weakly compact.  If $I$ is a closed ideal in $A$,
then $A/I$ is  weakly compact.  \end{lemma}

\begin{proof}  We leave this as an exercise for the reader.
\end{proof}

{\bf Remark.}   Similarly, if $A$ is an approximately unital ideal in its bidual
then so is any closed subalgebra, or quotient by a closed ideal (see \cite{Sharma}).

 \begin{corollary} \label{sssp} Suppose that $A$ is an operator algebra which is
a HSA in its bidual.  Then $A$ is semisimple (resp.\ semiprime) 
iff $A^{**}$ is semisimple (resp.\ semiprime).
\end{corollary}

\begin{proof}  The one direction follows from 
Theorem \ref{sspas} (resp.\ \cite[Proposition 2.5]{ABS}).
 If  $A$  is semisimple, let
 $0 \neq \eta \in J(A^{**})$.  Then $A \eta A \subset J(A^{**})
\cap A \subset J(A) = (0)$ (using
Theorem \ref{sspas}).  So $\eta = 0,$ and $A^{**}$ is semisimple.
If $A$ is semiprime, and if $J$ is an ideal in $A^{**}$ with
$J^2 = (0)$, then $(J \cap A)^2 = (0)$, so that $J \cap A = (0)$.
Hence $A J A  = (0)$ since $A J A \subset J \cap A$.  Since a cai of $A$
converges weak* to the identity of $A^{**}$ we deduce that $J = (0)$.
So $A^{**}$ is semiprime.
 \end{proof}

\begin{proposition} \label{idbid} If   $A$ is an operator algebra which is
a HSA in its bidual, and if $A$ has no ideals (resp.\ no closed 
ideals, no closed
ideals with a cai), then every ideal (resp.\ closed
ideal, closed ideal with a cai) in $A^{**}$ contains $A$.
\end{proposition}

\begin{proof}  If $J$ is a nontrivial
 ideal in $A^{**}$, then as in the proof of
Corollary \ref{sssp},  $A J A \subset J \cap A = A$ or $(0)$, and the 
latter is impossible.   Similarly for the closed
ideal case.  Similarly for the case of a closed
ideal $J$ with a cai, because 
by Corollary \ref{innp} the ideal $J \cap A$ of $A$ is also
a HSA in $A^{**}$, so has a cai.    
 \end{proof}

An operator algebra
$A$ with cai is {\em nc-discrete} if every right ideal which
has a left cai, is of the form $eA$ for a projection $e \in M(A)$.   Equivalently,  all the open projections are 
also closed (or equivalently are in $M(A)$).  The first part of the following was independently noticed recently 
in \cite{OS},
and no doubt by others:

\begin{proposition} \label{cafa}  A $C^*$-algebra
which is a HSA in its bidual, or is nc-discrete,
is an annihilator $C^*$-algebra.
\end{proposition}

\begin{proof}    One well known
 characterization of annihilator $C^*$-algebras
is that every commutative
 $C^*$-subalgebra $D$ has maximal ideal space which is
topologically discrete.  Thus the HSA case of the
proposition  follows by Lemma \ref{hsher}
and the fact that if a $C_0(K)$ space is an ideal in
its bidual, then $K$ is topologically discrete.
The nc-discrete case for $C^*$-algebras  is another well known
characterization of annihilator $C^*$-algebras.
    \end{proof}

\begin{proposition} \label{pehs}  If an operator algebra
$A$ is a HSA in its bidual (resp.\ is
compact), then so is $\Kdb_I(A)$ for any cardinal $I$.
Also, the $c_0$-direct sum of operator algebras which are
HSA's in their bidual
(resp.\ nc-discrete, $\Delta$-dual), is a HSA in its bidual (resp.\ is
nc-discrete, $\Delta$-dual).
\end{proposition}

\begin{proof}  We leave this as an exercise.
\end{proof}

{\bf Remark.}   We said earlier that compact
approximately unital Banach algebras are HSA's in their bidual.
We recall that a semisimple Banach algebra $A$ is
a modular annihilator algebra iff no element of $A$ has a nonzero
limit point in its spectrum \cite[Theorem 8.6.4]{Pal}), 
and iff  for every $a \in A$ multiplication on $A$ by $a$
is a Riesz operator (see \cite[Chapter 8]{Pal}).  If $A$ is also commutative then this is equivalent to
the Gelfand spectrum of $A$ being discrete \cite[p.\ 400]{LN}.  By  
\cite[Chapter 8]{Pal},  compact semisimple algebras are modular 
annihilator algebras.  
We note that any radical semiprime algebra is a modular annihilator algebra
by \cite[Theorem 8.7.2]{Pal}.   There are some interesting commutative radical  algebras in \cite{Dixr}
which are  modulator annihilator algebras, but 
they are probably not approximately unital nor are 
ideals in their bidual. One may ask if for
 algebras that are    HSA's in their  
 bidual (or even ideals in their bidual),
is the spectrum of every element finite or countable?   
We have examples of algebras which are HSA's in their
 bidual with elements having spectrum which
does have nonzero limit points (see Example \ref{ell2}).
Such algebras are not
 modular annihilator algebras, but are 
 {\em Duncan modular annihilator algebras} in the sense of  \cite[Chapter 8]{Pal}  (see also \cite{Dunc}).
Any semisimple operator algebra with the spectrum of any element finite or countable,
is a  Duncan modular annihilator algebra \cite{Pal}.   If $A$ is a commutative approximately unital 
 operator algebra
which is an  ideal in its bidual, one may ask if the spectrum of $A$ (eg.\
the set of characters of $A$) scattered?   In this case, and if $A$ is
not reflexive in the Banach space sense, then the spectrum of $A^{**}$ equals the 
one point compactification of the spectrum of $A$ (see Theorem \ref{asp} (4)).

In the converse direction, Duncan modular annihilator algebras, or
semisimple operator algebras with the spectrum of every
 element finite or countable,
need not be nc-discrete.  An example is the space $c$.
We are not sure if every (approximately unital)
semisimple modular annihilator operator
algebra is nc-discrete, or  is a HSA in its  bidual,
although this seems unlikely, even in the commutative case.   
.

\begin{theorem} \label{bvol} If an operator algebra $A$ is a HSA in $A^{**}$,
and if $\Delta(A)$ acts nondegenerately on $A$, then
$\Delta(A)^{**} = \Delta(A^{**}) = \Delta(M(A))$.   In particular, every
projection in $A^{**}$ is both open and closed.
\end{theorem}

\begin{proof}
That  $\Delta(A)$ acts nondegenerately on $A$ implies that $A$ has a positive
cai $(e_t)$ say.  If $A$ is a HSA in $A^{**}$ then
$e_t \eta e_t \in \Delta(A)$ for all $\eta \in \Delta(A^{**})_+$.
If we represent $A^{**}$ as a weak* closed subalgebra of $H$ containing
$I_H$, then $A$ is represented nondegenerately on $H$ (via \cite[Lemma 2.1.9]{BLM}
say), and so  $e_t \zeta \to \zeta$ for all $\zeta \in H$.  It follows
that $e_t \eta e_t \to \eta$ WOT, hence weak*.  Thus $\Delta(A^{**})
\subset \Delta(A)^{\perp \perp}$.  Since the converse inclusion is obvious
we have $\Delta(A)^{**} = \Delta(A^{**})$.  This equals
$\Delta(M(A))$ by \cite[Proposition 2.11]{ABS}.  For the last part,
note that $\Delta(A)$ is a HSA in its bidual by Lemma \ref{ncher},
and hence is an annihilator $C^*$-algebra
by Proposition \ref{cafa}.  Thus
any projection in $\Delta(A^{**}) = \Delta(A)^{**}$ is
open  and closed with respect to $\Delta(A)^{**}$ and hence also open with 
respect to $A$.  
\end{proof}

\begin{theorem} \label{hsnc}
If an operator algebra $A$ is a HSA in $A^{**}$ then $A$ is nc-discrete.
      \end{theorem}

\begin{proof}
We follow some ideas in
the proof of \cite[Proposition 5.1]{BHN}, which the reader might follow. 
By Lemma \ref{hsher},
 $\Delta(A)$ is a HSA in $\Delta(A)^{\perp \perp} = \Delta(A)^{**}$.
Hence $\Delta(A)$ is an annihilator $C^*$-algebra
by Proposition \ref{cafa}.

Next, let $p$ be an open projection
in $A^{**}$.  Suppose that $A$ is a subalgebra of a $C^*$-algebra $B$, generating $B$ as a $C^*$-algebra.
Then $A p A \subset A$, and by Cohen's factorization
$B p B = B A p A B \subset B$.
There is an increasing net $x_t \nearrow p$, with $x_t \in B$
for all $t$.  If $b \in B_+$,
then $b x_t b  \nearrow b p b \in B$.  Therefore, by Mazur's theorem,
replacing $x_t$ by
convex combinations of the $x_t$,     we may assume that
$b x_t b \to b p b$ in norm, and $0 \leq x_t \leq p$.
Then $\Vert \sqrt{p - x_t} b \Vert^2 = \Vert b (p - x_t) b \Vert \to 0$.
Hence $(p - x_t) b \to 0$, so that $p b \in B$.  Therefore $p$ is
a left multiplier of $B$.  However any projection which is a  left multiplier  is
a two-sided multiplier.  Consequently, $p A \in B \cap A^{\perp \perp} = A$,
so $p$  is
a left multiplier of $A$.  Similarly, $p$ is a right  multiplier of $A$, hence $p \in M(A)$.
\end{proof}

{\bf Remark.}  In this case there are bijective correspondences
between the right ideals in $A$ with left cai, HSA's in $A$,
and orthogonal projections in the multiplier
algebra $M(A)$.    This will follow from the previous theorem and the basic facts
about HSA's (see \cite[Section 2]{BHN}).

\begin{lemma} \label{ma}  Suppose that an approximately unital operator algebra $A$
is a HSA in its bidual, and that $\pi : A \to B(H)$ is a nondegenerate
completely isometric representation.   Then $A^{**} \cong \overline{\pi(A)}^{w*}$
as dual operator algebras.  Also, $A^{**}$ is an essential extension of $A$
 (that is,  every completely
 contractive linear map $T : A^{**} \to B(H)$, which
restricts to a complete isometry on $A$, is a complete 
isometry), and $A^{**}$ embeds as a unital subalgebra of a $C^*$-algebra
$I(A)$ which is an injective envelope of $A$. 
\end{lemma}

\begin{proof}  Most of this is essentially in \cite{Kan},
and follows standard ideas (see \cite[Section 2.6]{BLM}),
but for completeness we sketch a proof.
Define $QM_\pi(A) = \{ T \in B(H) :  \pi(A) T \pi(A)  \subset
\pi(A) \}$.  It is easy to see using \cite[Lemma 2.1.6]{BLM}
that $\pi(A) T \pi(A) = (0)$ implies that $T = 0$.
The canonical weak* continuous representation $\tilde{\pi} : A^{**} \to \overline{\pi(A)}^{w*}$
 maps into $QM_\pi(A)$, since $\pi(a) \tilde{\pi}(\eta) \pi(b) = \tilde{\pi}(a \eta b)
\in \pi(A)$.  Clearly $\tilde{\pi}$  is one-to-one, since $\tilde{\pi}(\eta) = 0$ implies
$a \eta b = 0$ for all $a, b \in A$, so that $\eta = 0$.
In fact $\tilde{\pi}(A^{**}) = QM_\pi(A)$.  To see this
suppose that  $T \in QM_\pi(A)$, $\Vert T \Vert \leq 1$.
Suppose that $\pi(e_t) T \pi(e_s) = \pi(a_{t,s})$ for each $s, t$.
For fixed $s$ the net  $(a_{t,s})$ has a subnet converging to $\eta_s \in
{\rm Ball}(A^{**})$.  Suppose that $\eta$ is a weak* limit point for
$(\eta_s)$ in ${\rm Ball}(A^{**})$.  Then $$\pi(a) \tilde{\pi}(\eta) \pi(b)
= \lim_\mu \pi(a) \tilde{\pi}(\eta_{s_\mu}) \pi(b) , \qquad a, b \in A ,$$
where this limit and the ones below are weak* limits.
However if $s = s_\mu$ is fixed, there is a net $(t_\nu)$ such that
$$\pi(a) \tilde{\pi}(\eta_{s}) \pi(b) = \lim_\nu
\pi(a) \tilde{\pi}(a_{t_\nu,s} \pi(b) =
\pi(a) \pi(e_{t_\nu}) T \pi(e_{s})  \pi(b) = \pi(a) T \pi(e_{s})  \pi(b) .$$
So $$\pi(a) \tilde{\pi}(\eta) \pi(b) = \lim_\mu \pi(a) \tilde{\pi}(\eta_{s_\mu}) \pi(b)
= \lim_\mu \pi(a) T \pi(e_{s_\mu})  \pi(b) = \pi(a) T \pi(b) .$$
So $\tilde{\pi}(\eta) = T$.  Thus $\tilde{\pi}$ is isometric, hence its range is
weak* closed, hence $QM_\pi(A) = \overline{\pi(A)}^{w*}$.
Once we that know $\tilde{\pi}$ is isometric,
applying  this in the setting of $M_n(A^{**}) \cong M_n(A)^{**}$ shows that
$\tilde{\pi}$ is completely isometric.

Suppose that $z \in {\rm Ball}(QM(A))$.  By the main theorem
in \cite{KP} (see also the quicker proof of \cite[Theorem 5.2]{BNmetricII}), $z$ corresponds to a
unique element $w \in {\rm Ball}(I(A))$ such that $a w b = a z b$ for all
$a, b \in A$.   This defines a contractive one-to-one unital map $\rho :
QM(A) \to I(A)$ which extends the identity map on $A$.  If this $w$ has norm $\kappa$
then $\Vert  e_t  z e_s \Vert \leq \kappa$ for each $s, t$, so that
$\Vert z \Vert \leq \kappa$.  Hence $\rho :
A^{**} \to I(A)$ is a unital isometry.  By the usual trick (using the isometry applied on
$M_n(A^{**}) = M_n(A)^{**}$,  and the fact that  $I(M_n(A)) = M_n(I(A))$ by 4.2.10 in \cite{BLM}), 
$\rho$ is a complete isometry.  Since $I(A)$ is  an essential extension of $A$,
we deduce that $A^{**}$ is an essential extension of $A$.   Now suppose that
$I(A^{**})$ is an injective envelope of $A^{**}$ containing $A^{**}$ as
a unital subalgebra.  Any complete contraction on $I(A^{**})$ which
restricts to a complete isometry on $A$,  must be a
complete isometry on $A^{**}$ by the last part, hence is a
complete isometry on $I(A^{**})$ by rigidity.   So $I(A^{**})$ is rigid for $A$,
hence is an 
injective envelope of $A$ with the desired property. 
     \end{proof}

{\bf Remark.}  The bulk of the first paragraph of the last proof
shows that $\tilde{\pi}$ is completely isometric.  However
this follows immediately from the second paragraph
(that fact there that $A^{**}$ is an essential extension).   Nonetheless
we felt it worthwhile to include a more elementary argument.

\begin{theorem} \label{asp}  Let $A$ be an operator algebra which
is a HSA in its bidual.
\begin{itemize}  \item [(1)]
$A$  is an Asplund space (that is, $A^*$ has the RNP).  
Also,
$A^*$ has no proper subspace that norms  $A$. 
\item [(2)]   $A^{**}$ is a {\em rigid extension} of $A$ in the 
Banach space category (that is, there is only one
 contractive linear map from $A^{**}$ to itself extending $I_A$).   
\item [(3)]   
Any surjective
linear complete isometry $A^{**} \to A^{**}$ is weak* continuous. 
\item [(4)]  There is a unique completely contractive extension $\tilde{\pi} 
: A^{**} \to B(H)$ of any nondegenerate completely contractive representation
$\pi : A \to B(H)$, namely the canonical weak* continuous extension.
In particular, every character of $A^{**}$ is weak* continuous.
\end{itemize} 
\end{theorem}

\begin{proof} 
(1) and (2) follow from \cite[Theorem 2.10]{BHN}, which 
says that such $A$ is `Hahn-Banach smooth',
and well known properties of `Hahn-Banach smooth' spaces
due to Godefroy and coauthors, and others (see e.g.\ \cite{GS}).

(3) \ This 
 follows from the Remark  at the end of Section 5
in \cite{ABS}.   

(4) \ In fact this is true even if $\pi$ is a linear  complete contraction with $\pi(e_t) \to I_H$ weak*. 
Note that if $\tilde{\pi}$ is a completely contractive extension of $\pi$, then for any unit vector
$\zeta \in H$, 
$\langle  \tilde{\pi}(\cdot) \zeta , \zeta \rangle$ is the unique (and hence necessarily weak* continuous)
extension from \cite[Theorem 2.10]{BHN}
of the state $\langle  \tilde{\pi}(\cdot) \zeta , \zeta \rangle$.  Thus $\langle  \tilde{\pi}(1) \zeta , \zeta \rangle
= 1$.  Hence $ \tilde{\pi}(1) = I$, and we may now appeal to  \cite[Proposition 2.11]{BHN}. 
 \end{proof}

{\bf Remark.}  
Being an Asplund space is
hereditary, so any closed subalgebra $C$ of an
operator algebra $A$ which is an Asplund space has
$\Delta(C)$ an Asplund space.  But does not  imply that
 $\Delta(C)$ is an annihilator $C^*$-algebra
(a $C^*$-algebra which is an Asplund space need not be annihilator,
certainly  $C_0(K)^*$ may be a separable $\ell^1$ space without
$K$ being discrete (consider $K$ the one point compactification of
$\Ndb$).  

As in \cite[Proposition 3.14]{Sharma} we obtain: 

\begin{corollary}  If an operator algebra $A$ is a HSA in its bidual, and if 
$A$ is
not reflexive, then it contains a copy of $c_0$.  Similarly, every approximately unital
subalgebra of $A$, and every quotient algebra of $A$, which  is
not reflexive,  contains a copy of $c_0$.
\end{corollary}

The following is a variant of
the `Wedderburn theorem' for operator algebras from \cite{ABS}:

\begin{corollary}  A separable operator algebra $A$ is $\sigma$-matricial in the sense
of \cite{ABS} iff $A$ is semiprime, a HSA in its bidual,  
and every  HSA $D$ in $A$ with ${\rm dim}(D) > 1$, contains a nonzero
projection which is not an identity for $D$.
\end{corollary}

\begin{proof}  Follows from Theorem 4.23 (vii) in \cite{ABS} together with
 Theorem \ref{hsnc}.  \end{proof}

The property of being nc-discrete also passes to subalgebras (and
to quotients by closed ideal having cai):

\begin{lemma} \label{ncher}  Let  $A$ be a nc-discrete
operator algebra with cai.  If $B$ is a closed subalgebra
of $A$ with a cai, then $B$ is nc-discrete.  If $I$ is a closed ideal in $A$, and if
$I$ has a cai, then $A/I$ is nc-discrete.
  \end{lemma}

\begin{proof}
  If $A$ is nc-discrete, a subalgebra of a $C^*$-algebra $B$,
 and $D$ is a closed approximately unital subalgebra
of $A$, with $p$ an open projection in $B^{**}$ which lies in
$D^{\perp \perp}$, then $p \in A^{\perp \perp}$, so $p \in M(A)$.
Then $p d \in D^{\perp \perp} \cap A = D$ for all $d \in D$.  Similarly $d p \in D$,
so $p \in M(D)$.  Thus $D$ is   nc-discrete.

Next, suppose that $A$ is nc-discrete, and that $I$ is an approximately unital ideal in $A$.
If $B$ is a $C^*$-algebra generated by $A$, then the
$C^*$-algebra generated in $B$ by $I$, is an ideal $J$ in $B$ \cite[Lemma 2.4]{BR}.
Let $q^\perp$ be the (central) support projection of $J$, which equals
the support projection of $I$. 
We make the identifications in
the proof of Lemma \ref{hsher} above.  Then 
$$A/I \subset A^{**}/I^{\perp \perp} \cong A^{**} q \subset B^{**} q .$$
The map $A/I \to B/J$ is a completely isometric embedding, since its
composition with the `canonical inclusion' $B/J  \subset B^{**} q$
is the complete isometry in the displayed equation above.
An open projection $e$ in $(A/I)^{**}$ which is
open with respect to
$A/I$, can thus be identified with
a projection $p \in B^{**} q$ such that there exists
a net $(x_t) \subset A$ with $x_t q \to p$ weak*, and $q x_t  = p x_t$
for all $t$.  By hypothesis, $q$ is open with respect to $A$, so that there
is a net $(y_s) \subset A$  with $y_s \to q$ weak* and $q y_s = y_s$ for all $s$.
Then $x_t y_s = x_t q y_s \to p q = p$, and $p x_t y_s  = q x_t y_s =
x_t y_s q = x_t y_s$.   It follows that $p$ is open in $A^{**}$.
Thus $p \in M(A)$.  Since $pa q = ap \in A \cap Aq$ for any $a \in A$,
it is clear that $e \in M(A/I)$. 
 \end{proof}

\begin{proposition} \label{dint}  If $A$ is a nc-discrete
approximately unital operator algebra, and is an integral domain,
then $A$ has no nontrivial r-ideals. \end{proposition}  \begin{proof}
This is clear:  The support
projection $p$ of any r-ideal is in $M(A)$, as is
$p^\perp$, and $A p^\perp pA = (0)$.
\end{proof}

\begin{proposition} \label{fa}  An ideal with cai
in a uniform algebra, which is  nc-discrete,  
is isometrically isomorphic to $c_0(I)$, for some set $I$.
\end{proposition}

\begin{proof}  
 If $A$ is a nc-discrete ideal with cai in a uniform  algebra, which we
can take to be $A^1$, then by Lemma \ref{ncher}
and Proposition \ref{cafa}, $\Delta(A)$ is a commutative
annihilator $C^*$-algebra.  Thus $\Delta(A)$ is densely spanned by its
minimal projections.  If $f$ is the sup of these minimal projections
in $\Delta(A)^{**}$, then obviously $f$ is open with respect
to $\Delta(A)$, hence open with respect to $A$, and  therefore is also
closed and is  in $M(A)$ (since $A$ is  nc-discrete).   
Then $J = A f^\perp$ is an ideal with cai in a uniform  algebra
and $J$ possesses no projections (for these would have
to be in $\Delta(A)$).  On the other hand, if $f \neq 1$
then $J$ has proper closed
1-regular ideals by Theorem 3.3 in \cite{ABS} if necessary,
and since $J$ is nc-discrete (by Lemma \ref{ncher}) it contains nontrivial
projections by Proposition 3.5  in \cite{ABS}. So $f = 1$.

If $e$ is a minimal projection in
$\Delta(A)$, then $eA$ is a uniform algebra containing no
nontrivial projections, hence containing no proper nonzero
closed ideals with cai (or else the support projection $f$ for
such an ideal would satisfy $f = fe \in A$, contradicting
minimality of $e$).  However every nontrivial uniform algebra contains
proper closed ideals with cai (for example those
associated with Choquet boundary points).  Thus $eA = \Cdb e$.
Hence $1_{A^1}$ is the sum of a family $\{ e_i : i \in I \}$ of
mutually orthogonal
algebraically minimal projections
in $A$, and so $A \cong c_0(I)$.    \end{proof}

In \cite{ABS} it is conjectured that $C^*$-algebras are exactly the  operator algebras
 satisfying conditions of the type: every closed
left ideal has a right cai.  
The following is a complement to Theorem 5.1 of \cite{ABS}.  The hypotheses
can be weakened further, we just state a simple representative form of the
result:

\begin{proposition}  A semisimple operator algebra $A$ which is nc-discrete,
such that every right ideal in $A$ has a left cai,
is an annihilator $C^*$-algebra.
\end{proposition}

\begin{proof}  It is obvious that $A$  is  
a left annihilator algebra. 
Thus $A$ has dense socle (see e.g.\ \cite[Chapter 8]{Pal}).
The rest is as in Theorem 5.1 of \cite{ABS}.
   \end{proof}

\section{Examples of operator algebras that are ideals in their biduals}

 In this section we list several examples answering natural questions that 
arise when investigating some of  the topics of this paper.

\subsection{A reflexive semisimple operator algebra}  \label{ell2}   A unitization
of some operator algebra structure on $\ell^2$ with pointwise product,
will be a unital reflexive commutative semisimple non-compact
operator algebra.   One such 
operator algebra structure on $\ell^2$ may be explicitly 
represented as follows.    Identify $\Ndb$ with two disjoint copies 
of $\Ndb$, and consider the span of $E_{1k} + E_{kk}$, with $k$ in the 
second copy of $\Ndb$, and $1$ here from the first copy.     This example
 is  not a modular annihilator algebra (because the canonical maximal
ideal has no nonzero annihilator), but it is 
a Duncan modular annihilator algebra in the sense of  \cite[Chapter 8]{Pal}, 
which is more general than a   modular annihilator algebra.  This example  
has no nontrivial r-ideals, since it is reflexive and has no 
nontrivial projections.   Its spectrum is the one point compactification
of $\Ndb$, which is a scattered topological space.

\subsection{A nc-discrete semisimple 
 operator algebra which is not a HSA in its bidual} \label{ell1}

It is known that $\ell^1$ with pointwise product is isomorphic to  
an operator algebra $A$ say (actually this may be done in many ways,
see e.g.\ \cite[Chapter 5]{BLM}),  and it is
semisimple.   Let $A^1$ be the unitization of $A$, then
$A^1$ is commutative, unital, semisimple, and it is  not an ideal in its 
bidual, since it is unital but 
not reflexive.   
We claim that the total number of orthogonal projections in $A^1$
is finite.  To see this, let $e_j$ be the minimal idempotents in $A$ 
coming from the canonical basis for $\ell^1$.
If $p=\lambda 1 + a$ is a projection in $A^1$ then
$\lambda$ is either 1 or 0, so either $p$ or $1-p$ is in $A$.
Also,  the only
projections in $A$ are finite sums of some of the $e_j$.
The sup of a finite family of orthogonal projections is
an orthogonal projection, so if there were infinitely 
many distinct  orthogonal projections in $A$ then
there would be arbitrarily large finite
sums of the $e_j$ represented in the operator algebra as
norm 1 projections.  This is impossible, because in $\ell^1$ the norm
of a sum of $n$  of the $e_j$ is $n$. So we have just finitely many
orthogonal projections in $A$, and the orthogonal projections in  
$A^1$ are these projections and their complements.  
Depending on the
 choice of representation, any finite ring of projections of $\ell^1$ 
can be the ring of orthogonal projections in $A$ (by basic similarity theory
such a finite family of idempotents are simultaneously similar to orthogonal 
projections); and the r-ideals
of $A^1$ include the unital 
ideals $pA$ and $(1-p)A^1$ for each projection $p$ in the chosen
ring. 

We claim that the just mentioned ideals are the only r-ideals of $A^1$, 
so that $A^1$ is nc-discrete.  
To see this, suppose  that  $J$ is an r-ideal. 
   Case 1:  $x + 1 \notin J$ for all $x \in A$.  In this case,
 $J$ is an ideal  in $A$.  Setting $E = \{ j \in \Ndb : e_j \in J \}$, it is easy to see
that $J = J_E$, where $J_E$ consists of the members of $A$ with `$j$th coordinate' 
zero for all $j \notin E$.   This is isomorphic to $\ell^1(E)$.
If $E$ is finite then $J$ is finite dimensional, hence $J = e A^1$
for a projection $e \in A^1$ as desired.
However  if $E$ is
infinite then $\ell^1(E)$ with pointwise product,
or equivalently $\ell^1$ or $A$, cannot have a bai.  Indeed if $A$ had 
a bai, then $A \subset B(A)$ via the regular representation, and then
the argument at the start of \cite[Section 4]{ABS} gives the 
contradiction that $(\sum_{k=1}^n \, e_j)$ is uniformly bounded.
      
Case 2: $x + 1 \in J$ for some $x \in A$.  By the argument above, $J \cap A = J_E$ for some set $E \subset \Ndb$.  If $y + 1 \in  J$ for some $y \in A$, then
$x-y \in J_E$.  It follows that $J = \Cdb (1+x) + J_E$.  
    If $j \notin E$ then since $e_j + x e_j \in J$
we must have $x e_j = -e_j$.   This can happen for at most a finite number of $j$; that is
$\Ndb \setminus E$ is finite.
Let $q$ be the sum of the $e_j$ for $j \in \Ndb \setminus E$.   Then $x + q  \in J_E \subset J$,
so that $1 - q = 1+x - (x+q) \in J$.  Since  $q(1+x) = 0$ we see that $J$ has an  identity 
$1-q$, which necessarily has norm $1$ since $J$ is approximately unital.      

Note that $A^1$ is not a modular annihilator algebra since $A$ has
no annihilator, but it is a Duncan modular annihilator algebra
in the sense of \cite[Section 8.6]{Pal}.  

\medskip

 The remainder of our examples are commutative and radical operator algebras.
In this connection we remark that there are quite a number of papers
on commmutative radical operator algebras  in the literature,
but most of these algebras are not approximately unital.
See for example  \cite{Dixr}
and \cite{PW}.   Indeed in  \cite{PW} and several 
related papers by Wogen, Larson, and others, one aim is to 
study an operator $T$ in terms
of the norm closed algebra oa$(T)$ generated by $T$, 
particularly  in cases where the latter algebra is radical.  
The following is one of the best studied examples:

\medskip

\subsection{The Volterra operator and a subquotient of the disk algebra}  Let $V$ be the Volterra operator on $[0,1]$.
Let $A_V$ be the norm closed algebra
generated by $V$.  We may write $A_V = {\rm oa}(T)$ for an operator $T$ with $\Vert I - T \Vert \leq 1$,
indeed let $T = I - (I+V)^{-1}$ (it is well known 
that the norm of $(I+V)^{-1}$ is $1$ and its spectrum
is $\{ 1 \}$).   We have oa$(T) \subset A_V$, and the converse inclusion 
holds since $V= (I-T)^{-1} T$.   This algebra has been studied extensively,
for example in \cite{PW} and \cite[Corollary 5.11]{Dav}.  
It is commutative, approximately unital,
compact, radical, and is an ideal in its bidual.
  Indeed, $M(A_V) = V' \cong A_V^{**}$
\cite[Corollary 5.11]{Dav}.
As we said in \cite[Section 5]{BRead},
has no r-ideals;
indeed all the closed ideals in this algebra
are known.
The algebra $A_V$ is nc-discrete
but not semiprime, in fact it has a dense ideal consisting 
of nilpotent elements.  

Jean Esterle suggested to the second  author in 2009 to look at the example 
$D  = B/[gB]$  where $B$ is the ideal
of functions in the disk algebra which vanish at the point 1,
and $g(z) = \exp((z+1)/(z-1))$.
Although $g$ is  not in the disk algebra, it is well known from the theory of inner
functions that $gB \subset B$, and that $gB$ is a closed proper ideal 
in $B$ (see top of p.\ 84 in \cite{Hof}).
By \cite[Proposition 2.3.4]{BLM}, $D$ is
a commutative operator algebra, and it has a cai since $B$ does.
In fact it turns out that $D$ is completely isometrically isomorphic to
the algebra $A_V$ above generated by the Volterra operator.  See  e.g.\  \cite{PW},
where  it is pointed out  that this 
leads to mutual insight into both the operator theory in $A_V$
and its commutant, and the function theory on the disk
associated with an interesting class of ideals of the disk algebra.
For example we see from this that $D$ is compact as a Banach algebra,
a fact that seems difficult to see by direct computations in $D$.   

 \subsection{Weighted convolution algebras which are ideals in their bidual} \label{weight} 

In this section we consider operator algebras formed from weighted  
convolution algebras $L^1(\Rdb_+,\omega)$.  By a {\em weight} we will mean a measurable function 
 $\omega: \Rdb_+ = [0,\infty)\to(0,
\infty)$ with $\omega(0) = 1$,  which is submultiplicative in the sense that
 $\omega(s+t) \leq \omega(s) \omega(t)$ for all $s, t \geq 0$.   Then for $1\le p<\infty$,
the set  $L^p(\omega)  = L^p(\Rdb_+,\omega)$ of  
 equivalence classes of measurable functions $f:\Rdb_+\to\Cdb$  such that
 $\Vert f \Vert_p =(\int_0^\infty \,
|f(t)|^p \, \omega(t)^p \, dt)^{1/p}<\infty$,  is a Banach space with norm $\Vert f \Vert_p$.  
As in \cite[Section 4.7]{Dal},
$L^1(\omega)$  with convolution product is a Banach algebra, and it is
radical iff $\lim_{t \to \infty} \,
\omega(t)^{\frac{1}{t}} = 0$.  Otherwise it is semisimple.    We write $R_x f$ for 
the right translation of $f$ by $x$; in 
other notation this is $\delta_x *  f$.  Similarly, we  write $L_x f$ for the left translation of $f$ by $x$.
As in \cite[Section 5]{BRead}, convolution induces a contractive homomorphism
$f \mapsto M_f$ from $L^1(\omega)$ into $B(L^2(\omega))$,  and we define ${\mathcal A} = {\mathcal
 A}(\omega)$ to be the norm closure of the set of  operators $M_f$ for $f \in  L^1(\omega)$.  This is an
operator algebra.  We write $\Vert \cdot \Vert_{\rm op}$ for the operator
norm on ${\mathcal A}(\omega)$ or more generally in $B(L^2(\omega))$.   Whenever we refer
below to `the operator norm' it is this one.   

It is known that  
$\frac{1}{\omega}$ is bounded on compact intervals, so the arguments in  \cite[Corollary 5.3]{BRead} 
work to show that ${\mathcal A}(\omega)$ is an integral domain,
and in particular is semiprime, and is not an annihilator algebra.       Being an  integral domain, it has 
no nontrivial idempotents, hence has zero socle.  If it is radical then it 
is a modular annihilator algebra in the sense of \cite{Pal}, by
\cite[Theorem 8.7.2]{Pal}. 
If  $\omega$  is right continuous at $0$ then there is a nonnegative cai  for 
$L^1(\omega)$, and hence for  ${\mathcal
 A}(\omega)$, consisting of constant multiples of 
characteristic functions of a sequence of compact intervals shrinking to 
$0$ (by e.g.\ \cite[4.7.41]{Dal}).   In this case, 
since $L^1(\omega) \cap L^2(\omega)$ is
dense in $L^2(\omega)$, it is clear that ${\mathcal A}(\omega)$ acts nondegenerately
on $L^2(\omega)$.  Hence if also  ${\mathcal A}(\omega)$ is an 
ideal in its bidual then ${\mathcal A}(\omega)^{**}$ may be identified with the
weak* closure of ${\mathcal A}(\omega)$ in $B(L^2(\omega))$ by
Lemma \ref{ma}, and  ${\mathcal A}(\omega)$ possesses no nontrivial r-ideals by Proposition \ref{dint} (and ${\mathcal A}(\omega)^{**}$ contains no nontrivial 
 projections).   We remark in passing that  \cite[Theorem 2.2]{BaDa} states that  if $L^1(\omega)$ contains a nonzero compact element then $\omega$ is radical.  

\begin{lemma} \label{dom}   If a weight $\omega:[0,\infty)\to(0,
\infty)$  is right continuous at $0$, and is not regulated at some $x > 0$ (that is   if 
$\frac{\omega(x+t)}{\omega(t)} \nrightarrow 0$), then ${\mathcal A}(\omega)$ is 
not compact.
\end{lemma}  \begin{proof}   Since $\limsup_{t \to \infty} \, \frac{\omega(x+t)}{\omega(t)} > 0$, there exist an $\epsilon > 0$ and an unbounded
increasing sequence of numbers  $(a_n)$ such that $\frac{\omega(x+a_n)}{\omega(a_n)} > \epsilon$.  Since $\omega$ is
continuous at $0$, it is bounded near $0$.   By submultiplicativity of $\omega$ it is bounded on any compact interval.
  If $y \in [0,x]$ then $\omega(x+a_n) \leq \omega(y+a_n) \, \omega(x-y)$, and it follows that (changing 
$\epsilon$ if necessary) \begin{equation} \label{rino}  \frac{1}{\epsilon} \leq \frac{\omega(y+a_n)}{\omega(a_n)} \leq \epsilon
, \qquad y \in [0,x]. \end{equation}  If $f$ is a nonzero nonnegative $C^\infty$ function supported on $[0,\frac{x}{3}]$, we claim that 
multiplication by $f$ is not compact on ${\mathcal A}(\omega)$ .  To see this define $f_n = \frac{1}{\omega(a_n)} R_{a_n} f$, which is supported on $[a_n, a_n + \frac{x}{3}]$. 
 We have  $$\Vert f_n \Vert_{L^1(\omega)}
= \int_0^{\frac{x}{3}} \, |f(t)| \, \frac{\omega(t+a_n)}{\omega(a_n)} \, dt \leq  \int_0^{\frac{x}{3}} \, |f(t)| \, \omega(t)  \, dt < \infty,$$  and  
$\Vert f_n \Vert_{L^1(\omega)} \in [\frac{\Vert f \Vert_{L^1(\Rdb_+)}}{\epsilon} ,
\epsilon \Vert f \Vert_{L^1(\Rdb_+)}]$
by (\ref{rino}).   Similarly 
$$\Vert f_n \Vert^2_{L^2(\omega)}
= \int_0^{\frac{x}{3}} \, |f(t)|^2 \, \frac{\omega(t+a_n)^2}{\omega(a_n)^2} \, dt \in 
[ \frac{\Vert f \Vert^2_{L^2(\Rdb_+)}}{\epsilon^2}  ,  \epsilon^2 \Vert f \Vert^2_{L^2(\Rdb_+)}]
.$$ 
Similarly, $\Vert f * f_n \Vert_{L^2(\omega)}  \geq
 \frac{\Vert f * f  \Vert^2_{L^2(\Rdb_+)}}{\epsilon}$. 
 Since $f * f_n$ is supported on $[a_n, \infty)$, this sequence converges
weakly to zero  in $L^2(\omega)$.  Thus $(f  *f_n)$ does not have a norm convergent subsequence in $L^2(\omega)$.
  Similarly,
$(f  *f_n * f)$ does not have a norm convergent subsequence in $L^2(\omega)$, and so 
$(f  *f_n)$ does not have a norm convergent subsequence in ${\mathcal A}(\omega)$.  Yet
$(f_n)$ is a norm bounded sequence in $L^1(\omega)$ and hence in ${\mathcal A}(\omega)$.  
\end{proof}  

\begin{corollary} \label{isc}  If $\omega$ is any weight, and if ${\mathcal A}(\omega)$ is compact
then $L^1(\omega)$  is compact.  \end{corollary}    \begin{proof}  By Lemma \ref{dom} we have that 
$\omega$ is regulated at all $x > 0$.  Now apply \cite[Theorem 2.7]{BaDa}.  \end{proof}  

{\bf Remark.}  The converse of Corollary \ref{isc} is false: $L^1(\omega)$  may be compact without 
 ${\mathcal A}(\omega)$ being compact.  This
follows from Theorem \ref{wcnc} below and \cite[Theorem 2.9]{BaDa}.  

\medskip

We say that the radical weight $\omega$ satisfies Domar's criterion
if the function $\eta(t)=-\log\omega(t)$ is a convex function
on $(0,\infty)$,
and for some $\epsilon >0$ we have $\eta(t)/t^{1+\epsilon}\to\infty$ as $t\to\infty$. An obvious example of such a weight is $\omega(t)=e^{-t^2}$. 
 In \cite[Section 5]{BRead} we studied ${\mathcal A}(\omega)$
if $\omega$ satisfies Domar's criterion.

\begin{proposition} \label{dom2}   If a  radical weight $\omega:[0,\infty)\to(0,
\infty)$ satisfies Domar's criterion  then
${\mathcal A}(\omega)$ is compact.
\end{proposition}  \begin{proof}   The collection of
compact operators on ${\mathcal A}$ is closed, hence it suffices to show that
$g \mapsto h * g$ is compact  on ${\mathcal A}$ for $h \in L^1(\omega)$.
Since the embedding of $L^1(\Rdb_+)$ in ${\mathcal A}$
is continuous, we may assume further that $h$ is bounded
(since simple functions are dense in $L^1(\omega)$), and $h$ has
support contained in $[\epsilon,N]$ say, for a fixed $\epsilon > 0$.
Indeed since $L^1(\omega)$ is an approximately unital
 Banach algebra,  the convolution is continuous
and $L^1(\omega) * L^1(\omega)
= L^1(\omega)$; hence we may assume that $h = f *g$
for bounded $f, g \in L^1(\omega)$ both with
support contained in $[\epsilon,N]$.  Clearly such $g \in L^2(\omega)$
since $g$ and $\omega$ are bounded, and so $a *g \in L^2(\omega)$
for $a \in {\rm Ball}({\mathcal A})$, and $\Vert a * g \Vert_{L^2(\omega)}
 \leq \Vert a \Vert_{{\mathcal A}} \Vert g \Vert_{L^2(\omega)}$.
By \cite[Corollary 5.6]{BRead}, if $\delta_{\frac{\epsilon}{2}} * a * g$
is a shift of $a * g$ to the right by $\frac{\epsilon}{2}$,  
then this is in $L^1(\omega)$, with norm there dominated by 
$\Vert a * g \Vert_{L^2(\omega)} \leq C \Vert g \Vert_{L^2(\omega)}$,
for a constant $C$.   On the other hand, if $f_2$ is a shift of 
$f$ to the left by $\frac{\epsilon}{2}$, then 
on any compact subinterval of $(0,\infty)$ the function $|f_2| \omega$ is
dominated by a constant times a left shift of $|f| \omega$,
 since $\omega$ is continuous.
Since  $f$ has compact support it follows that $f_2 \in L^1(\omega)$.    

Let $(x_n) \subset {\rm Ball}({\mathcal A})$.  Then $x_n * g \in
L^{2}(\omega)$, and $(\delta_{\frac{\epsilon}{2}} * (x_n * g))$ is a
bounded sequence
in $L^{1}(\omega)$, by the inequality involving $C$ in the last paragraph.  Since the latter algebra is compact by
\cite{BaDa}, there is a convergent subsequence of
$(f_{2} * (\delta_{\frac{\epsilon}{2}} * (x_n * g)))$ in $L^{1}(\omega)$,
and hence in ${\mathcal A}$.  However 
$f_{2} * (\delta_{\frac{\epsilon}{2}} * (x_n * g)) = f * (x_n * g) = h *
x_{n}$.   Thus multiplication by $h$ is compact on  ${\mathcal A}$ as
desired.  \end{proof}

Henceforth we take the weight $\omega$ to be a ``staircase weight", namely on
each interval $[n, n+1)$ we assume that $\omega$ is a constant $\frac{1}{a_n}$,
where $(a_n)$ is a strictly and rapidly  increasing sequence of positive integers with $a_0 = 1$.  
So long as $a_{n+m} \geq a_n a_m$ for $n, m \in \Ndb$ then   $\omega$ is a weight function,
and $L^1(\omega)$ and  ${\mathcal A}(\omega)$ are commutative Banach algebras with cai
(see e.g.\ \cite{BaDa,Dal}).  
It seems to be quite difficult to find examples of commutative Banach algebras 
(which are not reflexive in the Banach space sense), which are weakly compact 
 (hence ideals in their biduals by 
Lemma \ref{cohwc}) but not compact.  In fact this is mentioned as an open problem 
in \cite{Ulg}.  The following gives a commutative approximately unital operator algebra 
with this property.

\begin{theorem}  \label{wcnc}   Let $(\epsilon_n)$ be a decreasing null sequence, and let  $(a_n)$ be a sequence chosen with 
\begin{equation} \label{ri} a_0 = 1 \; , \; \; \; \; \; \; 
a_n \geq \max \{ \frac{2^k}{\epsilon_{n-k}} a_k \, a_{n-k} : k = 1, \cdots , n-1 \} .  \end{equation}
Then the operator algebra
${\mathcal A}(\omega)$ for the  associated staircase weight $\omega$ is not compact, but is
 weakly compact, and hence is an ideal in its bidual.   Also, ${\mathcal A}(\omega)$ is a 
commutative approximately unital radical operator algebra,
 which is not reflexive in the Banach space sense, and which is topologically singly generated.
\end{theorem}

 Clearly such a staircase weight
$\omega$ is not regulated at numbers in $(0,1)$, and hence  ${\mathcal A}(\omega)$  is
not compact by Lemma \ref{dom}.   To show that it is weakly compact, 
by the Eberlein-Smulian theorem it suffices to show that if $F \in {\mathcal A}$ and 
$(G_n)$ is a norm bounded sequence in ${\mathcal A}$, then 
$(F G_n)$ has a weakly convergent subsequence in $B(L^2(\omega))$.
Since the weakly compact operators are closed,
it is enough to show this in the case that  $F = \pi(f)$ and $G_n = \pi(g_n)$
for $f, g_n$  continuous functions of compact support on $\Rdb_+$ (since such functions
are dense in $L^1(\omega)$ by the usual arguments, 
and hence in ${\mathcal A}$).  Here $(G_n)$ is uniformly bounded in the 
operator norm, and so has a weak* convergent subsequence.  By passing 
to this subsequence we may assume that $G_n \to G$ weak*, and 
$F G_n \to F G$ weak*.  We will show that a subsequence $F G_{k_n} \to F G$ weakly.
     Set $H_n = G_n -G$.  We may, and will henceforth, assume that $\Vert H_n \Vert_{\rm op}
\leq 1$ for all $n$.  

Let $P_n$ denote the orthogonal projection onto the functions in
$L^2(\omega)$ supported on $[0,n]$, with $P_0 = 0$, and set $\Delta_n = P_{n+1} - P_n$. 
We have $P_n \pi(g) = P_n \pi(g) P_n$ for each $n$ and $g \in L^1((\omega)$, hence also
$P_n u = P_n u P_n$ for $u$ in ${\mathcal A}$ or in  $\bar{{\mathcal A}}^{w*}$.
 
\begin{lemma} \label{lefti}  Assume the hypotheses and notation of
Theorem {\rm \ref{wcnc}} and the discussion below it, and fix $m \in \Ndb$.
\begin{enumerate} \item [(1)]  Left multiplying  by $P_m \pi(g) = P_m \pi(g) 
P_m$ on ${\mathcal A}(\omega)$  is a compact operator on ${\mathcal A}(\omega)$ for  each  
 $g \in L^1(\omega)$.   
\item [(2)]  There exists $k_1 < k_2 < \cdots$ with 
$\Vert P_m \, F \, H_{k_n} \Vert_{\rm op} \to 0$ 
as $n \to \infty$.    
\end{enumerate}
\end{lemma}  \begin{proof}  
(1) \ Since the compact operators are closed
it suffices to prove (1) for a dense set of $g \in L^1(\omega)$, as in Proposition \ref{dom2},
and as in that proof we may assume that $g = g_1 * g_2$ for bounded
$g_1, g_2$ with support in $[\epsilon,N]$ where $0 < \epsilon < N$.  We may also assume that 
$g_1, g_2$ are continuous, by the same idea.   

For functions supported on $[0,m]$ the $L^1(\omega)$ norm
is equivalent to the  usual $L^1$ norm. We view $L^1([0,m]) \subset L^2([0,1])$ as a subspace
of $L^1(\omega) \cap L^2(\omega)$ in this way.
Suppose that $(a_n)$ is a bounded sequence
in ${\mathcal A}$.   Then $(g_2 * a_n)$ is a bounded sequence 
in $L^2(\omega)$, and so if $b_n = P_m(g_2 * a_n)$ then
$(b_n)$ is a bounded sequence in $L^1([0,m])$.    
  We then note that left multiplying by $P_m \pi(g_1)$
  may be viewed
as the operator  $T : L^1([0,m]) \to C([0,m]) : h \mapsto P_m (g_1 * h)$.
 Indeed $T$
 takes Ball$(L^1([0,m]))$ into a uniformly bounded
and equicontinuous subset of $C([0,m])$.
 By the Arzela-Ascoli theorem, $T$ is compact.
 Now $$T(b_n) = P_m (g_1 * P_m(g_2 * a_n))  = P_m (g_1 * g_2 * a_n) = P_m (g  * a_n) .$$
Since $T$ is compact and $(b_n)$ bounded, there will
be a subsequence $(P_m (g * a_{k_n}))$ that converges to a function in $C([0,m])$ in
the uniform norm.  Since the uniform norm on $C([0,m])$ dominates the 
$L^1$-norm, which is equivalent to the norm on $L^1(\omega)$, which dominates
the operator norm, we deduce that $(P_m (g  * a_{k_n}))$ converges in 
the operator norm too.    

(2) \ By (1), $(P_m F G_n)$ has a norm convergent 
subsequence, and the limit must be $P_m F G$ since  $F G_n  \to F G$ weak*. 
 The last  assertion is now obvious.   
\end{proof} 

\begin{lemma} \label{err1}  Under the hypotheses of 
Theorem {\rm \ref{wcnc}}, if $u \in \overline{{\mathcal A}(\omega)}^{w*}$ and $r \in \Ndb$,
$$\Vert u \, (I-P_r) - \sum_{m=r}^\infty \, \Delta_m \, u \, \Delta_m \Vert \leq \epsilon_r \Vert u \Vert.$$
Further, $\Delta_m \, u \, \Delta_m$ on ${\rm Ran}(\Delta_m)$ is
unitarily equivalent to $P_1  \, u \,  P_1$ on ${\rm Ran}(P_1)$, and hence
$\lim_{r \to \infty} \, \Vert u \, (I-P_r) \Vert = \Vert P_1  \, u \,  P_1 \Vert$.
All norms here are the operator norm.   \end{lemma}  

\begin{proof}    
To establish the inequality, we recall that $P_n u = P_n u P_n$, and so 
\begin{equation} \label{uone}  u \, (I-P_r) = \sum_{m=r}^\infty \, u \Delta_m = \sum_{m=r}^\infty \, \Delta_m  \, u \, \Delta_m
+ \sum_{m=r}^\infty \, \sum_{k=1}^\infty \, \Delta_{m+k}  \, u \, \Delta_m, \end{equation}
where these sums converge weak*.   We next estimate $\Vert \Delta_{m+k}  \, u \, \Delta_m \Vert$.
If $\eta \in {\rm Ran}(\Delta_m)$ then $\eta$ is supported on 
$[m,m+1]$, and so $\Vert \eta \Vert_{L^2(\omega)} = a_m^{-1} \Vert \eta \Vert_{L^2}$.
Then $\Delta_{m+k}  \, u \eta$ is supported on 
$[m+k ,m+k+1]$, and so $\Vert \Delta_{m+k}  \, u \eta \Vert_{L^2(\omega)} = a_{m+k}^{-1} \Vert 
\Delta_{m+k}  \, u \eta \Vert_{L^2}$.  Hence 
$\Vert \Delta_{m+k}  \, u \, \Delta_m \Vert$ is $\frac{a_m}{a_{m+k}}$ times the
norm of $\Delta_{m+k}  \, u \, \Delta_m$ as an operator on $L^2(\Rdb_+)$.
However the last norm equals  the 
norm of $\Delta_{k}  \, u \, \Delta_0$ as an operator on $L^2(\Rdb_+)$, since these operators are
unitarily equivalent on their supports (since it is easy to check that
$L_m \Delta_{m+k}  \, u \, \Delta_m R_m =  \Delta_k   \, u \, \Delta_0$, and the right shift $R_m$ by $m$
is an isometry with adjoint $L_m$). 
By the norm identity we just established for $\Vert \Delta_{m+k}  \, u \, \Delta_m \Vert$
in the case $m = 0$, we deduce that 
$$\Vert \Delta_{m+k}  \, u \, \Delta_m \Vert = \frac{a_m}{a_{m+k}} \, \frac{a_k}{a_{0}} \,
\Vert \Delta_{k}  \, u \, \Delta_0 \Vert \leq \frac{a_m a_k}  {a_{m+k}} \, \Vert u \Vert.$$  
  For varying $m$, the operators $\Delta_{m+k}  \, u \, \Delta_m$ have 
mutually orthogonal  left and right supports, and so 
$$\Vert \sum_{m=r}^\infty \, \Delta_{m+k}  \, u \, \Delta_m \Vert 
= \sup_{m \geq r} \, \Vert \Delta_{m+k}  \, u \, \Delta_m \Vert
\leq \Vert u \Vert \, \sup_{m \geq r} \, \frac{a_m a_k}  {a_{m+k}} 
\leq \Vert u \Vert \, \frac{\epsilon_r}{2^k} ,$$
the last inequality following from Equation (\ref{ri}).
 Hence $$\sum_{k=1}^\infty\, \Vert \sum_{m=r}^\infty \, \Delta_{m+k}  \, u \, \Delta_m \Vert 
\leq \Vert u \Vert \, \epsilon_r .$$ By straightforward operator theory,
by the observation above about mutually orthogonal left and right 
supports, we can interchange the double summation in Equation (\ref{uone})  to obtain
$$\Vert u \, (I-P_r) - \sum_{m=r}^\infty \, \Delta_m \, u \, \Delta_m \Vert \leq 
\sum_{k=1}^\infty\, \Vert \sum_{m=r}^\infty \, \Delta_{m+k}  \, u \, \Delta_m \Vert 
\leq \Vert u \Vert \, \epsilon_r,$$ as desired.   

Define $Jg(t) = a_m g(t-m)$ if $m \leq t \leq m+1$,
and $Jg(t) = 0$ otherwise.  Then $J$ is an isometry on Ran$(P_1)$, with kernel 
Ran$(I-P_1) = {\rm Ran}(P_1)^\perp$.  Thus $J$ is a partial isometry, and its final space is
clearly Ran$(\Delta_m)$.  The reader can check that $J P_1 \pi(f) P_1 J^* =
\Delta_m \pi(f) \Delta_m$ for $f \in L^1(\omega)$, and the same will be true
with $\pi(f)$ replaced by $u$, for $u$ in ${\mathcal A}$ or 
$\bar{{\mathcal A}}^{w*}$.
From this it is clear that $\Delta_m  \, u \,  \Delta_m$ is
unitarily equivalent to $P_1  \, u \,  P_1$ as stated.  The final 
assertion of the Lemma  is obvious since we are adding elements with mutually orthogonal supports.  \end{proof}

\begin{lemma} \label{err2}  Under the hypotheses of 
Theorem {\rm \ref{wcnc}}, and in the notation above, there is a sequence of  integers $1 = b_1 < c_1 < b_2 < 
c_2 < \cdots$ such that for each $n$,  
$$ \Vert (I -  P_{b_{n+1}}) F H_{c_n} \Vert < 2^{-n} \; , \; \;  \; \; \Vert P_{b_{n}} F H_{c_n} \Vert < 2^{-n} .$$
The norms here are the operator norm.   \end{lemma}  

\begin{proof}   Claim: left multiplication by $F P_m$ is a compact operator 
on $L^2(\omega)$.   Since $F$ has compact support there exists an $N$ such that
$F P_m = P_N F P_m$.   Suppose that $(g_n)$ is a bounded sequence
in $L^2(\omega)$.  Then $(P_m g)$ is a bounded sequence
in $L^1([0,m])$, and by the method in the proof of Lemma \ref{lefti} (1)
we have that $(P_N F P_m g)$ has a 
subsequence that converges to a function in $C([0,N])$ in
the uniform norm.  Since the uniform norm on $C([0,N])$ dominates a constant
multiple of the $L^2(\omega)$ norm, the subsequence converges in the
latter norm.  This proves the Claim.   It follows that $F H_n P_m = H_n F P_m$
is also compact on $L^2(\omega)$ for any fixed $n$.
Thus $(I-P_s) F H_n P_m \to 0$ in the operator norm as $s \to \infty$ for any fixed
$n, m$.

Assume that $1 = b_1 < c_1 < \cdots < b_k$ have already been chosen. 
 By Lemma \ref{lefti} (2), a subsequence of $(P_{b_{k}} F H_{n})$ converges in norm to $0$
by Lemma \ref{lefti} (2).  Thus  we may choose $c_k  > b_k$ with   $\Vert P_{b_{k}} F H_{c_k} \Vert < 
2^{-k-1}$.  Hence $\Vert P_{1} F H_{c_k} P_1 \Vert < 2^{-k-1}$.  By Lemma \ref{err1} we have
 $\lim_{r \to \infty} \, \Vert  F H_{c_k} \, (I-P_r) \Vert = \Vert P_1  \, F H_{c_k} \,  P_1 \Vert < 2^{-n-1}$.
Choose a particular $r$ with $\Vert  F H_{c_k} \, (I-P_r) \Vert < 2^{-k-1}$.
If $s > r$ we have $$\Vert (I -  P_{s}) F H_{c_k} \Vert < 2^{-k-1} +  \Vert (I -  P_{s}) F H_{c_k} P_r 
\Vert.$$
We know from the last line of 
the first paragraph of the proof that $\Vert (I -  P_{s}) F H_{c_k} P_r 
\Vert \to 0$ as $s \to \infty$, so we may choose  $b_{k+1} > c_k$ with 
$\Vert (I -  P_{b_{k+1}}) F H_{c_k} \Vert < 2^{-k}$.  This completes the inductive step.
\end{proof}

 \begin{proof} (Completion of the proof of Theorem \ref{wcnc}.)   If we choose $b_n, c_n$ as in 
Lemma \ref{err2}  we have $\Vert F H_{c_n} + (P_{b_n} - P_{b_{n+1}}) F H_{c_n} \Vert < 2^{1-n}$ for all $n$.
Thus to show that $F H_{c_n} \to 0$ weakly, which concludes our proof that 
${\mathcal A}(\omega)$ is weakly compact, it is enough that  $R_n =  (P_{b_n} - P_{b_{n+1}}) 
 F H_{c_n} \to 0$ weakly.  But this is clear since any   
bounded family  $(R_n) $ of operators on a Hilbert space $H$
with mutually orthogonal ranges converges weakly to $0$.  Indeed
if  $\sum_{n=1}^\infty \, R_n R_n^* \leq 1$ then
$R_n \to 0$  weakly.  To see this 
note that for any state $\varphi$ on $B(H)$,
we have $\sum_{n=1}^\infty \, \varphi(R_n R_n^*) \leq 1$.
Thus  by the Cauchy-Schwarz inequality, $|\varphi(R_n)|^2 \leq \varphi(R_n R_n^*)  \to 0$.

The assertion about the second dual follows from 
Lemma \ref{cohwc}.
To see that ${\mathcal A}(\omega)$ is radical, we first note
 that by mathematical induction on the case $k=n-1$ in Equation  (\ref{ri}),
we have  $a_n \geq \frac{a_1^n}{\epsilon_1^{n-1}} \,  2^{1 + 2 + 3 + \cdots + (n-1)}$.   Thus $a_n^{\frac{1}{n}} \to \infty$ and 
so $\lim_{t \to \infty} \, \omega(t)^{\frac{1}{t}} = 0$.   
Finally,  ${\mathcal A}(\omega)$ is not reflexive since if it 
were then it would have an identity of norm $1$, which cannot happen for radical algebras.  
It is topologically singly generated by the constant function $1$,  by 
\cite[Theorem 4.7.26]{Dal}.  \end{proof}

\section{The diagonal of a quotient algebra}  We remark that it is easy to see that if $A$
and  $B$ are closed subalgebras of $B(H)$ then $\Delta(A \cap B) = \Delta(A) \cap \Delta(B)$.

\begin{proposition} \label{ini} If $J$ is an inner ideal in
an operator algebra $A$
(i.e.\  $JAJ \subset J$), then $J \cap \Delta(A) = \Delta(J)$. 
\end{proposition}

\begin{proof}   It is trivial that $\Delta(J)$ is a subalgebra of $J \cap \Delta(A)$.
 Conversely, if $J A J \subset J$, then
$(J \cap \Delta(A)) \Delta(A) (J \cap \Delta(A)) \subset J \cap \Delta(A)$.  So
$J \cap \Delta(A)$  is a HSA in a C*-algebra, hence it is selfadjoint.  So
$x \in J \cap \Delta(A)$  implies that $x^* \in J \cap \Delta(A) \subset J$,
and so $x \in \Delta(J)$.
\end{proof}

We will use the fact that the diagonal of an ideal $J$ in $A$ is an
ideal in $\Delta(A)$ if it is nonzero. Indeed, $\Delta(J) = \Delta(A) \cap J$
 and $\Delta(J) \Delta(A) \subset \Delta(A) \cap (JA) \subset \Delta(A) \cap J = \Delta(J)$.
Similarly, since $J$ is a two-sided ideal, $\Delta(A) \Delta(J) \subset \Delta(J)$.
We sometimes will silently use this fact.   However it is not always true that $\Delta(A/J) \cong \Delta(A)/\Delta(J)$, if
$J$ is an approximately unital ideal in an approximately unital operator
algebra $A$.  A counterexample is given by the ideal of functions
in the disk algebra vanishing at two points on the circle, inside
the ideal of functions vanishing at one point.  Most of the rest of
this section is an attempt to understand this phenomenon, and
to give some conditions ensuring that $\Delta(A/J) \cong \Delta(A)/\Delta(J)$.

\begin{proposition} \label{Deltaquotient} Let $A$ be an approximately unital operator algebra
and let $J$ be an  ideal in $A$.
Then $\Delta(A) /\Delta(J) \subset \Delta(A/J)$.
\end{proposition}
\begin{proof} Let $u : A \to A/J$ be the canonical complete quotient map defined as $u(x)= x + J$. The restriction of $u$ to $\Delta(A) \subset A$, $u'$, is a complete contraction.
Since $u'$ is a contractive homomorphism, it maps into $\Delta(A/J)$ by 2.1.2 of
\cite{BLM}. Hence, we have a completely contractive map $u'=u|_{\Delta(A)}: \Delta(A) \to \Delta(A/J)$, where $\Ker{u'}= \Delta(A) \cap J = \Delta(J)$.
By the fact mentioned above the proposition,
$\Delta(J)$ is an approximately unital ideal in $\Delta(A)$, and we deduce that 
$\Delta(A)/\Delta(J) \subset \Delta(A/J)$ completely isometrically.
\end{proof}

\begin{lemma} \label{delex}  Let  $J$ be an approximately unital ideal
with positive cai in an approximately unital operator algebra $A$.  Then $\Delta(A/J)
 \cong \Delta(A)/\Delta(J)$ canonically
 iff every positive element of $A/J$ lifts to an element $b \in A$
such that $b \Delta(J) \subset \Delta(J)$.
\end{lemma}

\begin{proof}  One direction is obvious since as we said
earlier, $\Delta(J)$ is an ideal in $\Delta(A)$.

For the other direction, let $p$ be the support projection of
$J$.  The element $b$ in the statement satisfies $b = bp + b p^\perp$.
Now the canonical isometric homomorphism $A/J \to A^{**} p^\perp \subset  
A^{**}$ takes the diagonal of $A/J$ into $\Delta(A^{**})$.
Thus $b p^\perp \in \Delta(A^{**})$.  If $(e_t)$ is a cai for
$\Delta(J)$ then $b e_t \in \Delta(J)$, and in the limit we also have $bp \in
\Delta(J)^{\perp \perp} \subset \Delta(A^{**})$ (the latter
since $\Delta(J) \subset \Delta(A^{**})$).
  So $b \in \Delta(A^{**}) \cap A = \Delta(A)$.  From this
the result is evident.  \end{proof}

For any operator algebra $A$, the diagonal $\Delta(A)$ acts nondegenerately on $A$ if and only if $A$ has a positive cai, and if and only if $1_{\Delta(A)^{\perp \perp}}=1_{A^{**}}$. The latter is equivalent to $1_{A^{**}} \in \Delta(A)^{\perp \perp}$. Hence, we may use the statements `$\Delta(A)$ acts nondegenerately on $A$' and `$A$ has a positive cai' interchangeably.

\medskip

{\bf Remark.}  
If $\Delta(A)$ acts nondegenerately on $A$, then this does not imply that 
$\Delta(J)$ acts nondegenerately on $J$ if $J$ is an ideal of $A$.
 To see this, take any approximately unital operator algebra $J$ such that $\Delta(J)$ does not act nondegenerately on $J$. Then, $J$ is an ideal in $A = M(J)$, and $\Delta(A)$ acts nondegenerately on $A$. However, if $\Delta(A)$ acts nondegenerately on $A$, then  $\Delta(A/J)$ acts nondegenerately on $A/J$.  Indeed,
 it is fairly evident by e.g.\ 2.1.2 in \cite{BLM} that if $J$ is an ideal in an operator algebra $A$,
and if $A$ has a positive cai, then $A/J$ has a positive cai.

\begin{proposition} \label{a2} Let $A$ be an operator algebra.  If $J$ is a left
ideal in $A$ with a
selfadjoint right cai, then 
$\Delta(J) = \Delta(A) \cap J$ is a left ideal in $\Delta(A)$,
$\Delta(A) + J$ is closed, and $(J \cap \Delta(A))^{\perp \perp}
= J^{\perp \perp}  \cap \Delta(A)^{\perp \perp}$.
\end{proposition}

\begin{proof}  The first statement is evident 
from Proposition \ref{ini}.
 Write $(e_\lambda)$ for the selfadjoint right cai of $J$.
If $r \in \Delta(A)$ then $r e_\lambda \in \Delta(A)
\cap J$.  This is what is needed to make
the idea in the proof of \cite[Proposition 2.4]{Dixon}
work, as in the proof of 
Corollary \ref{innp}, giving that $\Delta(A) + J$ is closed,
By the proof of \cite[Lemma 5.29 and Appendix A.1.5]{BZ},
$(\Delta(A) \cap J)^{\perp \perp}= (\Delta(A)^{\perp} + J^{\perp})^{\perp}
= \Delta(A)^{\perp \perp} \cap J^{\perp \perp}$.
\end{proof} 

\begin{proposition} \label{Deltabidual} Let $A$ be an approximately unital operator algebra such that $\Delta(A^{**})=\Delta(A)^{**}$. If $J$ is an ideal in $A$ that contains a positive cai, then $\Delta(J^{\perp \perp})=\Delta(J)^{\perp \perp}$ and $\Delta((A/J)^{**})=(\Delta(A)/\Delta(J))^{**}$.
\end{proposition}
\begin{proof} By Proposition \ref{a2}, $\Delta(J)^{\perp \perp} =
\Delta(A)^{\perp \perp} \cap J^{\perp \perp}$. Clearly  
$\Delta(A^{**}) \cap J^{\perp \perp} = \Delta(J^{\perp \perp})$ by Proposition
\ref{ini}, and so $\Delta(J^{\perp \perp})=\Delta(J)^{\perp \perp}$.
 Let $p \in A^{**}$ be the support projection of $J$. Then,
\[ \Delta((A/J)^{**}) = \Delta(A^{**}(1-p))= \Delta(A^{**})(1-p) = \Delta(A)^{**}(1-p) = (\Delta(A)/\Delta(J))^{**}, \]
where we have used Remark 2.10 (ii) in \cite{ABS}.  \end{proof}

\begin{corollary} \label{usi}  Let $A$ be an operator algebra with a positive cai that
is a HSA in its
bidual. If $J$ is an ideal in $A$ that possesses a positive cai,
then $\Delta(A/J) = \Delta(A)/\Delta(J)$.
\end{corollary} \begin{proof}  Since $A$ is a HSA in its bidual and it contains a positive cai, we have $\Delta(A^{**}) =\Delta(A)^{**}$ by Theorem \ref{bvol}.
Also, $A/J$ is a HSA in its bidual $(A/J)^{**}$ by Lemma \ref{hsher},
 and it can easily be seen by e.g.\ 2.1.2 in \cite{BLM} to have a positive cai.
Hence  $\Delta((A/J)^{**}) = \Delta(A/J)^{**}$ by Theorem \ref{bvol}.
 Moreover, since $A/J$ is a HSA in its bidual, 
 $\Delta(A/J)$ is an annihilator $C^*$-algebra. 
We know by Proposition \ref{Deltaquotient} that $\Delta(A) / \Delta(J) \subset \Delta(A/J)$ completely isometrically.  Hence $\Delta(A)/\Delta(J)$ is an annihilator $C^*$-algebra as well.
Using Proposition \ref{Deltabidual}, we have $(\Delta(A) / \Delta(J))^{**} = \Delta((A/J)^{**}) = \Delta(A/J)^{**}$.
Hence, $\Delta(A)/\Delta(J) = \Delta(A/J)$.
\end{proof}

We end with another result on the diagonal related to Corollary
\ref{innp}:

\begin{proposition}   If $A$ is an approximately unital operator algebra, if
$D$ is a HSA in $A$ and if $J$ is an approximately unital ideal in $A$, then 
$\Delta (D \cap J)^{\perp \perp} = \Delta (D)^{\perp \perp} \cap \Delta(J)^{\perp \perp}$.  If $D$ and $J$ have positive cais, then  $D \cap J$ has a positive cai as well.

\end{proposition}
\begin{proof}  A modification of the proof  of
Proposition \ref{a2} or Corollary \ref{innp} shows that
$\Delta (D \cap J)^{\perp \perp} = \Delta (D)^{\perp \perp} \cap \Delta(J)^{\perp \perp}$. 

 If $D$ and $J$ have positive cais, then
 \[ 1_{(D \cap J)^{\perp \perp}} = 1_{D^{\perp \perp} 
\cap J^{\perp \perp}} = 1_{D^{\perp \perp}} 1_{J^{\perp \perp}} \in \Delta(D)^{\perp \perp} \cap \Delta(J)^{\perp \perp} = \Delta( D \cap J)^{\perp \perp}. \]
That is, $1_{(D \cap J)^{\perp \perp}} \in \Delta(D \cap J)^{\perp \perp}$. Hence, $D \cap J$ has a positive cai.
\end{proof}

\end{document}